\pdfoutput=1

\documentclass{amsart}
\usepackage{latexsym,amsfonts,amssymb,amsthm,amsmath}
\usepackage[colorlinks=true,hidelinks]{hyperref}
\usepackage{amscd}
\usepackage{mathrsfs,multicol}
\usepackage[all]{xy}  
\usepackage{enumerate}

\usepackage{stmaryrd}
\usepackage{color}
\usepackage{imakeidx}
\usepackage{tikz}
\usepackage{tikz-cd}
\usepackage{pdflscape}
\usepackage{cleveref}
\usepackage[vcentermath,enableskew]{youngtab}
\usepackage{centernot}
\usepackage{multirow}
\usepackage{caption}
\usepackage{graphicx}
\usepackage{array}
\usepackage{mathtools} 
\usepackage{adjustbox} 
    \AtBeginDocument{} 
\usepackage{dsfont} 
\usepackage[utf8]{inputenc}
\usepackage[T1]{fontenc}
\usepackage{microtype}
\usepackage{lmodern}
\usepackage{hyperref}

\usepackage{tikz-cd}  
\usepackage{quiver}
 \DeclareFontFamily{U}{min}{}
 \DeclareFontShape{U}{min}{m}{n}{<-> udmj30}{}
 
 \usepackage{cleveref} 

\usepackage[textwidth=3cm, textsize=small, colorinlistoftodos]
{todonotes}

\usetikzlibrary{positioning} 
\usetikzlibrary{decorations.pathreplacing}
\usetikzlibrary{arrows, decorations.markings} 
\pgfarrowsdeclarecombine{twotriang}{twotriang}{stealth}{stealth}
{stealth}{stealth}
\tikzset{mono/.style={>-stealth}} 
\tikzset{epi/.style={-twotriang}} 
\tikzset{arrow/.style={->}}
\tikzset{arrowshorter/.style={->, shorten <=2pt, shorten >=2pt}}
\tikzset{twoarrowlonger/.style={double,double distance=1.5pt,
shorten <=5pt,shorten >=6pt,
decoration={markings,mark=at position -4pt with {\arrow[scale=1.75]{>}}},
preaction={decorate}}} 

\tikzset{mapstikz/.style={-stealth, 
decoration={markings,mark=at position 0pt with {\arrow[scale=0.5]{|}}}, preaction={decorate}}}
\tikzset{dot/.style={circle,draw,fill,inner sep=1pt}}

\theoremstyle{plain}

\newtheorem{prop}[equation]{Proposition}
\newtheorem{lemma}[equation]{Lemma}

\theoremstyle{definition}
\newtheorem{definition}[equation]{Definition}
\newtheorem{example}[equation]{Example}

\newtheorem{remark}[equation]{Remark}

\newtheorem{proposition}[equation]{Proposition}

\numberwithin{equation}{section}

\newcommand{\Ar}{\operatorname{Ar}}

\newcommand{\Cat}{\mathcal Cat}

\newcommand{\coker}{\operatorname{coker}}
\newcommand{\colim}{\operatorname{colim}}
\newcommand{\Deltaop}{\Delta^{\op}}

\newcommand{\Fun}{\operatorname{Fun}}

\newcommand{\ob}{\operatorname{ob}}
\newcommand{\op}{\operatorname{op}}

\newcommand{\Set}{\mathcal Set}

\newcommand{\SSets}{\mathcal{SS}ets}
\newcommand{\Top}{\mathcal Top}

\makeatletter
\renewcommand*\env@matrix[1][\arraystretch]{%
  \edef\arraystretch{#1}%
  \hskip -\arraycolsep
  \let\@ifnextchar\new@ifnextchar
  \array{*\c@MaxMatrixCols c}}
\makeatother

\begin{document}









\begin{center}$2$-Segal maps associated to a category with cofibrations

Tanner Carawan
\end{center}

\tableofcontents \addtocontents{toc}
{\protect\thispagestyle{myheadings}\markright{}}

\newpage

\section{Abstract} \label{abstract}

\thispagestyle{myheadings}


Waldhausen's $S_\bullet$-construction gives a way to define the algebraic $K$-theory space of a category with cofibrations. Specifically, the $K$-theory space of a category with cofibrations $\mathcal{C}$ can be defined as the loop space of the realization of the simplicial topological space $|iS_\bullet \mathcal{C} |$. Dyckerhoff and Kapranov observed that if $\mathcal{C}$ is chosen to be a proto-exact category, then this simplicial topological space is 2-Segal. A natural question is then what variants of this $S_\bullet$-construction give 2-Segal spaces. We find that for $|iS_\bullet \mathcal{C}|$, $S_\bullet\mathcal{C}$, $wS_\bullet\mathcal{C}$, and the simplicial set whose $n$th level is the set of isomorphism classes of $S_\bullet\mathcal{C}$, there are certain $2$-Segal maps which are always equivalences. However for all of these simplicial objects, none of the rest of the $2$-Segal maps have to be equivalences. We also reduce the question of whether $|wS_\bullet \mathcal{C}|$ is $2$-Segal in nice cases to the question of whether a simpler simplicial space is $2$-Segal. Additionally, we give a sufficient condition for $S_\bullet \mathcal{C}$ to be $2$-Segal. Along the way we introduce the notion of a generated category with cofibrations and provide an example where the levelwise realization of a simplicial category which is not $2$-Segal is $2$-Segal.

\section{Introduction} \label{C1}
\thispagestyle{myheadings}

Algebraic $K$-theory is an area of math where groups indexed by integers are assigned to rings, exact categories, and other objects.
Algebraic $K$-theory groups often enjoy theorems like Additivity, Localization, and Approximation and this area of math has connections to number theory and topology. 

In the process of extending the collection of objects for which $K$-groups are defined to have an algebraic $K$-theory of spaces, Waldhausen gave the $S_\bullet$-construction which can be applied to an exact category or more generally to a category with cofibrations.


Let $\mathcal{E}$ be an exact category. The $S_\bullet$-construction produces a sequence of spaces, the $n$th one of which is the realization of a category $iS_n\mathcal{E}$ of diagrams built from the exact sequences of $\mathcal{E}$. This sequence of spaces determines a space called the $K$-theory space, and the homotopy groups of this space are the $K$-theory groups. 
    
Although the homotopy-theoretic properties of Waldhausen's $S_\bullet$-construction are still being investigated, two independent groups of authors, one consisting of Dyckerhoff and Kapranov \cite{DyckerhoffKapranov12} and another consisting of Gálvez-Carrillo, Kock, and Tonks \cite{decomposition}, noticed that if $\mathcal{E}$ is an exact category, then the sequence of spaces $|iS_n\mathcal{E}|$ has the structure of a $2$-Segal space. 

A sequence of spaces $X_n$, $n \geq 0$, form a $2$-Segal space if for every order preserving map $\alpha \colon \{0 \leq 1 \leq \cdots \leq m\} \to \{0 \leq 1 \leq \cdots \leq n\}$ (written $\alpha \colon [m] \to [n]$), there are compatibly chosen maps $X_n \to X_m$ such that certain induced maps $X_k \to X_2 \times^h_{X_1} X_2 \times^h_{X_1} \cdots \times^h_{X_1} X_2$, called $2$-Segal maps, are weak equivalences. We think of a $2$-Segal space as having a space of objects $X_0$, a space of morphisms $X_1$, and an up-to-homotopy sometimes-defined multivalued composition given by a span 
$X_1 \times^h_{X_0} X_1 \leftarrow X_2 \to X_1$; notice that the indices on the left are $0$ and $1$ and not $1$ and $2$. The $2$-Segal maps being weak equivalences says that this composition is associative up-to-homotopy. 

One benefit of the $S_\bullet$-construction is that it applies to a huge class of categories, categories with cofibrations. For $\mathcal{C}$ a category with cofibrations this work addresses whether the $2$-Segal maps from the $|iS_n\mathcal{C}|$ to certain iterated homotopy pullbacks are also weak homotopy equivalences. We find that a certain collection of them are, but none of the rest have to be. We also give a sufficient condition for $S_\bullet\mathcal{C}$ to be a $2$-Segal space. Finally, we do the preceeding with several variants of the $S_\bullet$-construction. Along the way, we introduce the notion of a generated category with cofibrations and provide an example where the levelwise realization of a non $2$-Segal space is $2$-Segal.  

The work is structured as follows. We begin with the inputs to variants of the $S_\bullet$-construction in Section 3. Then in Section 4 we review the $S_\bullet$-construction itself. Section 5 is a technical aside on homotopy limits. In Section 6 we prove our main result for a discrete version $\text{iso}(s_\bullet \mathcal{C})$ of the $S_\bullet$-construction. Then in Section 7 we prove our main result for a topological version $|iS_\bullet \mathcal{C}|$ of the $S_\bullet$-construction and we give our progress on analyzing the $2$-Segal maps of $|wS_\bullet\mathcal{C}|$ when $\mathcal{C}$ is a Waldhausen category. In Section 8 we prove our main result for the categorical $S_\bullet$-constructions $S_\bullet\mathcal{C}$ and $wS_\bullet\mathcal{C}$. In the final non-appendix section we provide a sufficient condition for $S_\bullet \mathcal{C}$ to be $2$-Segal. In the appendix we show that one of our central definitions is equivalent to a preexisting one, namely that left and right $2$-Segal are the same as lower and upper $2$-Segal. 

\section{Waldhausen Categories} \label{C2}
\thispagestyle{myheadings}

In this section, we review categories with cofibrations and Waldhausen categories because they are the inputs to the $S_\bullet$-construction. We also introduce a dual notion to that of a category with cofibrations, a category with fibrations.  

\subsection{Categories with cofibrations}

We begin by defining a category with cofibrations. In this section we include relevant terminology and facts about categories with cofibrations as well. 

\begin{definition}
 \cite[1.1]{Waldhausen85} A {category with cofibrations} is a category $\mathcal{C}$ with a zero object $0$ and a distinguished class of morphisms called the \emph{cofibrations}. These cofibrations and zero object are such that for any object
$A$ of $\mathcal{C}$, the unique map $0 \rightarrow A$ is a cofibration, every isomorphism is a cofibration, the composite of cofibrations is a cofibration, and the pushout of a cofibration always exists and is a cofibration. 
\end{definition}
We indicate that a morphism is a cofibration by drawing our arrows with a tail ($\rightarrowtail$).
\begin{example}
A motivating, and for our purpose central, example is $\mathcal{R}_f$, the category of finite based CW-complexes with specified CW structure and based cellular maps  \cite[9.1.4]{Weibel13}. In $\mathcal{R}_f$ the cofibrations are the cellular inclusions and the $0$ object is the singleton space.
\end{example}

One reason for having categories with cofibrations is to be able to consider cofibration sequences.



\begin{definition} 
A \emph{cofibration sequence} is a cofibration followed by its cokernel.  We denote the cokernel object of $i \colon A \rightarrowtail B$ by $\coker(i)$ and the cokernel morphism by $B \twoheadrightarrow \coker(i).$ We sometimes denote the cokernel object by $B/A$.
\end{definition}

\begin{definition}
A functor $F: \mathcal{C} \rightarrow \mathcal{D}$ between categories with cofibrations is called \emph{exact} when 
it preserves zero objects, cofibrations, and pushouts of cofibrations.  
\end{definition}


\begin{definition}
We say $\mathcal{A}$ is a \emph{subcategory with cofibrations} of a category with cofibrations $\mathcal{C}$ if the inclusion of $\mathcal{A}$ into $\mathcal{C}$ is an exact functor, and the cofibrations in $\mathcal{A}$ are the cofibrations of $\mathcal{C}$ with source and target in $\mathcal{A}$ and with cokernel in $\mathcal{A}$.
\end{definition}

The category whose objects are small categories with cofibrations and whose morphisms are exact functors, is called \textit{the category of categories with cofibrations} and we denote it by $\textit{Cof}$.

Now we present the dual notion to a category with cofibrations. By dual notion, we mean that the opposite category of a category with cofibrations is a category with fibrations and vice versa. Everything that follows in later sections about categories with cofibrations has a dual version for categories with fibrations.

\begin{definition}
A \emph{category with fibrations} is a category $\mathcal{C}$ with a zero object and a distinguished class of morphisms called the \emph{fibrations}. These fibrations and zero object are such that for any object
$A$ of $\mathcal{C}$, the unique map $A \rightarrow 0$ is a fibration, every isomorphism is a fibration, the composite of fibrations is a fibration, and the pullback of a fibration always exists and is a fibration. 
\end{definition}

We indicate that a morphism is a fibration by drawing our arrows with two heads ($\twoheadrightarrow$). This notation does not mean however that the morphism is the cokernel of a cofibration. Similarly, in a category with fibrations we denote that a morphism is the kernel of a fibration by drawing our arrow with a tail ($\rightarrowtail$). We call a fibration preceded by its kernel a fibration sequence. 

We denote the category whose objects are categories with fibrations by $\textit{Fib}$.

\begin{example}
Consider the category $\text{Vect}_k$ of $k$-vector spaces and linear maps. The $0$-dimensional $k$-vector space together with the collection of surjective linear maps give this category the structure of a category with fibrations. 
\end{example}

\subsection{Waldhausen categories}

 Waldhausen categories are categories with cofibrations, together with another class of morphisms meant to behave like weak homotopy equivalences.

\begin{definition}\cite[1.2]{Waldhausen85}
A \emph{Waldhausen category} is a category with cofibrations that has a class $w$ of morphisms called 
\emph{weak equivalences}, such that all isomorphisms are weak equivalences, $w$ is closed under 
composition, and pushouts along cofibrations are unique up to weak equivalence. To be precise, this last condition means that if in the commutative diagram 
\[\begin{tikzcd}[sep = small,ampersand replacement=\&]
	B \&\& A \&\& C \\
	\\
	{B'} \&\& {A'} \&\& {C'}
	\arrow[tail, from=1-3, to=1-1]
	\arrow[from=1-3, to=1-5]
	\arrow[from=1-1, to=3-1]
	\arrow[from=1-3, to=3-3]
	\arrow[from=1-5, to=3-5]
	\arrow[tail, from=3-3, to=3-1]
	\arrow[from=3-3, to=3-5]
\end{tikzcd}\] the vertical arrows arrows are in $w$, then the induced map $B\cup_A C \to B'\cup_{A'}C'$ is in $w$. We write $w\mathcal{C}$ for the subcategory of $\mathcal{C}$ whose collection of morphisms is $w$ and co $w\mathcal{C}$ for the subcategory of $\mathcal{C}$ whose collection of morphisms is the cofibrations in $w$. The maps in co $w\mathcal{C}$ are called \emph{acyclic cofibrations}.

\end{definition}


\begin{example}
Let $\mathcal{C}$ be a category with cofibrations. Then we can make $\mathcal{C}$ a Waldhausen category by taking the weak equivalences to be the isomorphisms.
\end{example}

\begin{example}
Another example is given by $\mathcal{R}_f$. Here the weak equivalences are the weak homotopy equivalences. 
\end{example}


There are some commonly ascribed additional properties a Waldhausen category might have. To define these properties it is convenient to give names to some categories one can construct given a category with cofibrations. 

\begin{definition}
Let $\mathcal{C}$ be a category. Its \emph{arrow category} $\Ar(\mathcal{C})$ is a category whose objects are morphisms in $\mathcal{C}$ and whose morphisms are commutative squares in $\mathcal{C}$.
\end{definition}

We denote a morphism from $f \colon X \to Y$ to $f' \colon X' \to Y'$, depicted as \[\begin{tikzcd}[sep = small]
	X && {X'} \\
	\\
	Y && {Y'}
	\arrow["f"', from=1-1, to=3-1]
	\arrow["{f'}", from=1-3, to=3-3]
	\arrow["\phi", from=1-1, to=1-3]
	\arrow["\psi"', from=3-1, to=3-3]
\end{tikzcd}\] by $(\phi, \psi) \colon f \to f'$. Two functors that one always has are the source and target functors $S\colon \text{Ar}(\mathcal{C}) \to \mathcal{C}$ and $T \colon \text{Ar}(\mathcal{C}) \to \mathcal{C}$ which take morphisms in $\mathcal{C}$ to their sources and targets, respectively. If $\mathcal{C}$ is a category with cofibrations, then $\Ar(\mathcal{C})$ is a category with cofibrations, where a cofibration is given by a pair of cofibrations in $\mathcal{C}$. The category $\Ar(\mathcal{C})$ is Waldhausen with weak equivalences given by pairs of weak equivalences.  

\begin{definition}
The full subcategory of $\Ar(\mathcal{C})$ on the cofibrations of $\mathcal{C}$, which we denote by $F_1 \mathcal{C}$, is a category with cofibrations the maps $(A \rightarrowtail B) \to (A' \rightarrowtail B')$ such that $A \rightarrowtail A'$, $B \rightarrowtail B'$, and $A' \cup_A B \rightarrowtail B'$ are cofibrations in $\mathcal{C}$. The category $F_1 \mathcal{C}$ is Waldhausen with weak equivalences the component-wise weak equivalences.
\end{definition}

\begin{definition}
A Waldhausen category is \emph{saturated} if whenever $gf$ is a weak equivalence, then $f$ is a weak equivalence if and only if $g$ is a weak equivalence. 
\end{definition}

Motivated by mapping cylinders of spaces, one has the following definition. 

\begin{definition}
Let $\mathcal{C}$ be a Waldhausen category.  
A \emph{mapping cylinder functor} on $\mathcal{C}$ is a functor from $\text{Ar}(\mathcal{C})$ to the category of diagrams in $\mathcal{C}$ taking $f \colon A \to B$ to a diagram 

\[\begin{tikzcd}[sep = small]
	{S(f)} && {M(f)} && {T(f)} \\
	\\
	&& {T(f)}
	\arrow["{i_f}", from=1-1, to=1-3]
	\arrow["{j_f}"', from=1-5, to=1-3]
	\arrow["f"', from=1-1, to=3-3]
	\arrow["r_f"', from=1-3, to=3-3]
	\arrow["{=}", from=1-5, to=3-3]
\end{tikzcd}\] where $M(f)$ is called the cylinder of $f$. The induced natural transformations, $i\colon S \to M$, $j \colon T \to M$, and $r\colon M \to T$ are called the \emph{top inclusion}, \emph{bottom inclusion}, and \emph{retract}, respectively. The functor $M$ and the natural transformations must satisfy the following axioms. 

\begin{itemize}

\item (MapCyl 1) The functor $i \sqcup j \colon \Ar(\mathcal{C}) \to F_1C$ is exact. 
	
\item (MapCyl 2) For every object $Y$ in $\mathcal{C}$, $M(0 \to Y) = Y$ with $r_{0 \to Y}$ and
$j_{0 \to Y}$ the identity. 

\end{itemize}

We say that $\mathcal{C}$ satisfies the mapping cylinder axiom if it has a mapping cylinder such that for all $f\colon X \to Y$ in $\mathcal{C}$, $r_f \colon M(f) \to Y$ is a weak equivalence.
\end{definition}

Notice that these axioms imply the following.

\begin{itemize}
\item (MapCyl 3) For every morphism $(\phi, \psi) \colon f \to f'$ in $\text{Ar}(\mathcal{C})$, 
$M((\phi, \psi)) \colon M(f) \to M(f')$ is a weak equivalence if $\phi$ and $\psi$ are weak equivalences.

\item (MapCyl 4) For every morphism $(\phi, \psi) \colon f \to f'$ in $\text{Ar}(\mathcal{C})$, 
$M((\phi, \psi)) \colon M(f) \to M(f')$ is a cofibration if $\phi$ and $\psi$ are cofibrations.

\item (MapCyl 5) For every morphism $(\phi, \psi) \colon f \to f'$ in $\text{Ar}(\mathcal{C})$, the diagram \[\begin{tikzcd}[sep = small]
	{X \sqcup Y} && {M(f)} && Y \\
	\\
	{X' \sqcup Y'} && {M(f')} && {Y'}
	\arrow["\phi\sqcup\psi"', from=1-1, to=3-1]
	\arrow["{i_f \sqcup j_f}", from=1-1, to=1-3]
	\arrow["{i_{f'} \sqcup j_{f'}}"', from=3-1, to=3-3]
	\arrow["{M((\phi,\psi))}"', from=1-3, to=3-3]
	\arrow["{r_{f'}}"', from=3-3, to=3-5]
	\arrow["\psi", from=1-5, to=3-5]
	\arrow["{r_f}", from=1-3, to=1-5]
\end{tikzcd}\] commutes. Furthermore, if $\phi$ and $\psi$ are cofibrations, then the induced map 
$(X' \sqcup Y') \bigsqcup_{X \sqcup Y} M(f) \to M(f')$ is a cofibration. 

\item(MapCyl 6) Given a diagram of the form 
\[\begin{tikzcd}[sep = small]
	A && B \\
	\\
	{A'} && {B'}
	\arrow["f", from=1-1, to=1-3]
	\arrow[tail, from=1-1, to=3-1]
	\arrow["{f'}"', from=3-1, to=3-3]
	\arrow[tail, from=1-3, to=3-3]
\end{tikzcd}\] in $\mathcal{C}$, there is a commutative diagram 

\[\begin{tikzcd}[sep = small]
	{A\sqcup B} && {M(f)} \\
	\\
	{A' \sqcup B'} && {M(f')} \\
	\\
	{A'/A \sqcup B'/B} && {M(f'/f)}.
	\arrow[tail, from=1-1, to=1-3]
	\arrow[tail, from=3-1, to=3-3]
	\arrow[tail, from=1-1, to=3-1]
	\arrow[tail, from=1-3, to=3-3]
	\arrow[two heads, from=3-1, to=5-1]
	\arrow[two heads, from=3-3, to=5-3]
	\arrow[tail, from=5-1, to=5-3]
\end{tikzcd}\] where $f'/f$ is the map from $A'/A$ to $B'/B$ induced by $f$ and $f'$.

\end{itemize}

\begin{proposition}

Let $\mathcal{C}$ be a saturated Waldhausen category together with a functor $M \colon \textrm{Ar}(\mathcal{C}) \to \mathcal{C}$ and natural transformations $r \colon M \to T$, $i \colon S \to M$, and $j \colon T \to M$, where $T$ is the target functor and $S$ is the source functor. Suppose that the $f$ component of $r$, $r_f$, is a weak equivalence when $f$ is a weak equivalence and that $M$ satisfies (MapCyl1), (MapCyl3) and (MapCyl4). Then the inclusion $\iota \colon$co $w\mathcal{C} \to w\mathcal{C}$ is a homotopy equivalence \cite{Weibel13}[IV, Exercise 8.15].
\end{proposition}

\begin{proof}
This proof comes from \cite{Warner}[Section 18, Proposition 3]. By Quillen's Theorem A \cite[1]{Quillen73}, it suffices to show that $\iota$ is initial. We must show that for every $Y \in ob(w\mathcal{C})$, $\iota / Y$ is contractible. An object of $\iota / Y$ is an object $X$ in $\mathcal{C}$ together with a weak equivalence $f$ from $X$ to $Y$. A morphism in $\iota / Y$ from $(X, f)$ to $(Z, g)$ is an acyclic cofibration $h \colon X \to Z$ such that $gh = f$. We define a functor $m \colon \iota /Y \to \iota /Y$ that sends $(X, f) \to (M(f), r_f)$. The fact that $m$ is a functor follows from $r$ being a natural transformation, $r_f$ being a weak equivalence, and $M$ satisfying (MapCyl3) and (MapCyl4). Now $i$ defines a natural transformation $id_{\iota/Y} \to m$ using (MapCyl1), the assumption that $\mathcal{C}$ is saturated, and that $i$ is a natural transformation from $S$ to $M$. Similarly, $j$ defines a natural transformation from the constant functor with image object $(id_y, Y)$ to $m$. We therefore have homotopies from the identity on $|\iota/Y|$ to $|m|$ and from a constant map from $\iota/Y \to \iota /Y$ to $|m|$. We conclude that $\iota/Y$ is contractible, and since $Y$ was arbitrary, $|\iota|$ is a homotopy equivalence.  
\end{proof}

\section{The $S_\bullet$-construction and 2-Segal objects} \label{C3}
\thispagestyle{myheadings}

\subsection{The $S_\bullet$-construction}
A sequence of cofibrations in $\mathcal{C}$ together with choices of cokernels gives an object in the category $S_n\mathcal{C}$. The categories $S_n\mathcal{C}$ taken together with certain functors between them form the simplicial category $S_\bullet\mathcal{C}$. This construction was given in Waldhausen's paper \cite[\S 1]{Waldhausen85} en route to creating an algebraic $K$-theory of spaces. We now review 
this construction and $2$-Segal objects. 

\begin{definition}
The category $\Delta$, called the \emph{simplex category}, has objects the ordered sets $[n]\colon = \{0 \leq 1 \leq \cdots \leq n\}$ and morphisms order-preserving maps. 
\end{definition}

We note that if we view $[n]$ as a poset category, then $\Delta$ can be viewed as the full subcategory of the category of small categories $\Cat$, on those categories of the form $[n]$ for some nonnegative integer $n$. It is often helpful to know that every morphism in 
$\Delta$ is the composite of face and degeneracy maps. The face maps $\delta^n_i \colon [n - 1] \to [n]$ are the injective maps "skipping $i$", and the degeneracy maps $\sigma^n_i \colon [n + 1] \to [n]$ are the surjective maps "hitting $i$ twice". Given a functor $X \colon \Delta^{op} \to \mathcal{D}$, a simplicial object in $\mathcal{D}$, we denote $X([n])$ by $X_n$ and we denote the image of $\delta^n_i$ by $d_i$ and the image of $\sigma^n_i$ by $s_i$. 

\begin{definition}
Given categories $\mathcal{A}$ and $\mathcal{B}$, the \emph{category of functors} $\Fun(\mathcal{A}, \mathcal{B})$ has functors from $\mathcal{A}$ to $\mathcal{B}$ as its objects and natural transformations as its morphisms. 
\end{definition}

\begin{example}
The arrow category of $\mathcal{C}$ (Definition 3.11) may be identified with $\Fun([1],\mathcal{C})$. 
\end{example}

\begin{example}
Fix $n \in \mathbb{N}$. Viewing $[n]$ as a poset category we consider $\Ar[n]$. Any morphism $i \to j$ in $[n]$ may be identified with the pair $(i,j)$ where $i \leq j$. Any commutative square
\[\begin{tikzcd}[sep = small]
	i && j \\
	\\
	{i'} && {j'}
	\arrow[from=1-1, to=1-3]
	\arrow[from=1-1, to=3-1]
	\arrow[from=3-1, to=3-3]
	\arrow[from=1-3, to=3-3]
\end{tikzcd}\] in $[n]$ may be identified with two pairs, $(i,j)$ and $(i', j')$, where $i \leq j$ and $i' \leq j'$, and with $i \leq i'$ and $j \leq j'$. In other words, $\Ar[n]$ may be identified with the poset category whose objects are pairs $(i,j)$ with $0 \leq i \leq j \leq n$, and where $(i,j) \leq (i', j')$ if and only if $i \leq i'$ and $j \leq j'$. The diagram 
\[\begin{tikzcd}[sep = small]
	&&&&&&&& {(n,n)} \\
	\\
	&&&&&&&& \vdots \\
	\\
	&&&& {(2,2)} && \cdots && {(2.n)} \\
	\\
	&& {(1,1)} && {(1,2)} && \cdots && {(1,n)} \\
	\\
	{(0,0)} && {(0,1)} && {(0,2)} && \cdots && {(0,n)} \\
	&&&&&& {}
	\arrow[from=9-1, to=9-3]
	\arrow[from=9-3, to=9-5]
	\arrow[from=9-5, to=9-7]
	\arrow[from=9-7, to=9-9]
	\arrow[from=9-3, to=7-3]
	\arrow[from=7-3, to=7-5]
	\arrow[from=9-5, to=7-5]
	\arrow[from=7-5, to=7-7]
	\arrow[from=7-7, to=7-9]
	\arrow[from=7-5, to=5-5]
	\arrow[from=5-5, to=5-7]
	\arrow[from=5-7, to=5-9]
	\arrow[from=9-9, to=7-9]
	\arrow[from=7-9, to=5-9]
	\arrow[from=5-9, to=3-9]
	\arrow[from=3-9, to=1-9]
\end{tikzcd}\]
gives a picture of this category. 
\end{example}

\begin{definition}
For $\mathcal{C}$ a category with cofibrations, $S_n\mathcal{C}$ is the full subcategory of the category of functors  $A \colon \Ar[n] \rightarrow \mathcal{C}, (i,j) \mapsto A_{i,j}$ on those functors with $A_{j,j} = 0$ for every $j$, and $A_{i,j} \rightarrowtail A_{i,k} \twoheadrightarrow A_{j,k}$ a cofibration sequence for every $i < j < k$. For $\mathcal{C}$ a category with fibrations $S_n\mathcal{C}$ is the full subcategory of $\Fun(\Ar[n], \mathcal{C}$) on functors $A$ with $A_{j,j} = 0$ for every $j$ and $A_{i,j} \rightarrowtail A_{i,k} \twoheadrightarrow A_{j,k}$ a fibration sequence for every $i < j < k$. 
\end{definition}

Waldhausen \cite[1.3]{Waldhausen85}  showed that $S_n \mathcal{C}$ is a category with cofibrations whenever $\mathcal{C}$ is. 

\begin{prop}
The assignment $[n] \mapsto S_n\mathcal{C}$ defines a functor $S_\bullet\mathcal{C} \colon  \Deltaop \rightarrow \textit{Cof}.$
\end{prop}

\begin{proof}
Objects of $S_m\mathcal{C}$ are also functors $\Ar [m] \to \mathcal{C}$. The map $[n] \to [m]$ in $\Delta$ induces a map $\Ar[n] \to \Ar[m]$ which in turn induces a map 
$S_m\mathcal{C} \to S_n\mathcal{C}$ given by precomposition. Therefore our assignment is a functor from $\Delta^{op} \to \Cat$. We have mentioned that $S_n\mathcal{C}$ is a category with cofibrations so it remains to check that the images of morphisms in $\Delta^{op}$ are exact. See \cite[1.1]{Waldhausen85}. 

\end{proof}

Similarly, if $\mathcal{C}$ is a category with fibrations then $[n] \mapsto S_n\mathcal{C}$ defines a functor $\Delta^{op} \to \textit{Fib}$. 

\begin{definition}
There is a forgetful functor $\Cat \rightarrow \Set$ that sends a small category to its set of objects and a functor to the corresponding function on sets of objects. We denote the composite $\Delta^{\op} \xrightarrow{S_\bullet\mathcal{C}} \Cat \rightarrow \Set$ by $s_\bullet \mathcal{C}$.
\end{definition}

Each simplicial map $\alpha \colon [m] \to [n]$  induces a map between the set of isomorphism classes of the objects of $S_n\mathcal{C}$ and the set of isomorphism classes of the objects of $S_m\mathcal{C}$

\begin{definition}
We denote by $\text{iso}(s_\bullet \mathcal{C})$ the simplicial set that
takes $[n]$ to $\text{iso}(s_n\mathcal{C})$, the set of isomorphism classes of objects in $S_n\mathcal{C}$. 
\end{definition}

Since a Waldhausen category is also a category with cofibrations, we may apply the $\text{iso}(s_\bullet -$),$s_\bullet$-, and $S_\bullet$-constructions. For each $n$, $S_n \mathcal{C}$ is a Waldhausen category where a weak equivalence is a component-wise weak equivalence. The subcategory of weak equivalences in $S_n\mathcal{C}$ is denoted by $wS_n\mathcal{C}$. Since $\mathcal{C}$ is also a Waldhausen category with weak equivalences the isomorphisms, we can consider $S_n\mathcal{C}$ as a Waldhausen category whenever $\mathcal{C}$ is a category with cofibrations. 

\begin{proposition}
\cite[1.6.1]{Waldhausen85}
If $\mathcal{C}$ is a saturated Waldhausen category satisfying the mapping cylinder axiom, then $S_n\mathcal{C}$ satisfies the hypotheses of Proposition 3.15. 
\end{proposition}

\subsection{2-Segal objects}
The goal of this paper broadly is to see what structure $S_\bullet \mathcal{C}$ and its variants have, given the properties we assume for $\mathcal{C}$. Some structure that $S_\bullet \mathcal{C}$ might have is that of a $2$-Segal category. We review $2$-Segal objects here. The definition of a $2$-Segal object was inspired by the definition of a $1$-Segal object so we start there.

\begin{definition}
A simplicial set $X$ is a \emph{$1$-Segal set} if the map $X_n \rightarrow X_1 \times_{X_0} X_1 \times_{X_0} X_1 \cdots \times_{X_0} X_1$ induced by $\{i,i + 1\} \hookrightarrow [n]$ for all $i \in [n - 1]$ is a bijection for all $n \geq 2$ \cite[1.1]{Boors18}.
\end{definition}

A $1$-Segal set is always isomorphic to the nerve of a category, so we think of a $1$-Segal set as corresponding to a category with objects the $0$-simplices and $1$-simplices the morphisms \cite[2.2.2]{DyckerhoffKapranov12}. Therefore in some sense, one can study categories by studying $1$-Segal sets. If, instead of strict composition, we wanted to model composition up to homotopy we could replace pullbacks with homotopy pullbacks and change the requirement of having a bijection to having a weak equivalence.
One framework where we have homotopy pullbacks and weak equivalences is that of model categories.

\begin{definition}
For a model category $\mathcal{D}$, a simplicial object in $\mathcal{D}$ is \emph{$1$-Segal} if the map $X_n \rightarrow X_1 \times^h_{X_0} X_1 \times^h_{X_0} X_1 \cdots \times^h_{X_0} X_1$ induced by $\{i,i + 1\} \hookrightarrow [n]$ for $i \in [n - 1]$ is a weak equivalence where the right-hand side is an iterated homotopy pullback. 
\end{definition}

\begin{remark}
We could instead take $\mathcal{D}$ to be an $(\infty, 1)$-category or any other category with a notion of homotopy limits and weak equivalences. 
\end{remark}

The maps in the above definitions are called \emph{$1$-Segal maps}. We can describe them in a different way that better suggests the generalization to $2$-Segal maps. In a sufficiently nice model category $\mathcal{D}$ there are derived mapping objects $\text{Map}^h(K,X)$ for all simplicial sets $K$ and simplicial objects $X$ in $\mathcal{D}$ \cite[1.1]{Boors18}. The derived mapping objects are such that if $W\colon A \to \SSets$ is a diagram, then $\text{Map}^h(\colim_A W_a , X)$ is weakly equivalent to $\text{holim}_A \text{Map}^h(W_a, X)$ and $\text{Map}^h(\Delta[n], X)$ is weakly equivalent to $X_n$ for all $n$.    Now fix $n$ and consider the diagram 
\[\begin{tikzcd}[sep = small]
	{\Delta[0]} && {\Delta[1]} && {\Delta[0]} && {\Delta[1]} && {\Delta[0]} && \cdots && {\Delta[0]} \\
	\\
	\\
	\\
	&&&&&& {\Delta[n].}
	\arrow["{0 \to 0}", from=1-1, to=1-3]
	\arrow["{0 \to 1}"', from=1-5, to=1-3]
	\arrow["{0 \to 0}", from=1-5, to=1-7]
	\arrow["{0 \to 1}"', from=1-9, to=1-7]
	\arrow["{0 \to 0}", from=1-9, to=1-11]
	\arrow["{0 \to 1}"', from=1-13, to=1-11]
	\arrow["{0 \to 0}"', from=1-1, to=5-7]
	\arrow[from=1-5, to=5-7]
	\arrow["{0 \to 0, 1 \to 1}"{description}, from=1-3, to=5-7]
	\arrow["{0 \to 1, 1 \to 2}"{description}, from=1-7, to=5-7]
	\arrow["{0 \to 2}", from=1-9, to=5-7]
	\arrow["{0 \to n}", from=1-13, to=5-7]
\end{tikzcd}\]

From the diagram above there is an induced map 
$\Delta[1]\sqcup_{\Delta[0]}\Delta[1]\sqcup_{\Delta[0]} \cdots \sqcup_{\Delta[0]}\Delta[1] \to \Delta[n]$ that we think of as the inclusion of the spine of $\Delta[n]$ into $\Delta[n]$. It turns out that applying $\text{Map}^h(-,X)$ yields a map fitting into the commutative square

\[\begin{tikzcd}[sep = small]
	{\text{Map}^h(\Delta[1]\sqcup_{\Delta[0]}\Delta[1]\sqcup_{\Delta[0]} \cdots \sqcup_{\Delta[0]}\Delta[1],X)} && {\text{Map}^h(\Delta[n], X)} \\
	\\
	{X_1 \times_{X_0}^h X_1 \times_{X_0}^h \cdots \times_{X_0}^h X_1} && {X_n.}
	\arrow[from=1-3, to=1-1]
	\arrow["\simeq"', from=1-1, to=3-1]
	\arrow[from=3-3, to=3-1]
	\arrow["\simeq"', from=3-3, to=1-3]
\end{tikzcd}\] Therefore the $1$-Segal maps are induced by the inclusion of a $1$-dimensional subsimplicial set of $\Delta[n]$ made up of ($n-1$) $1$-simplices glued to each other along $0$-simplices. To obtain $2$-Segal maps we 
can start with subsimplicial sets of $\Delta[n]$ consisting of ($n - 2$) standard $2$-simplices 
glued to each other along $1$-simplices. These $2$-dimensional subsimplicial sets can be unfolded and stretched into a 
triangulation of a regular $n$-gon. We number the vertices of polygons to give a way of specifying triangulations. 

\begin{definition}
We say that a polygon has its vertices \emph{cyclically labeled} if it is possible to read the vertices in counterclockwise order as
$0,1,2 \ldots, n$ for some $n$. We write $P_n$ for the regular $(n + 1)$-gon with cyclically labeled vertices.
\end{definition}

\begin{figure}[h!]
\definecolor{zzttqq}{rgb}{0.6,0.2,0}

\definecolor{zzttqq}{rgb}{0.6,0.2,0}
\begin{tikzpicture}[line cap=round,line join=round,>=triangle 45,x=1cm,y=1cm]

\fill[line width=2pt,color=zzttqq,fill=zzttqq,fill opacity=0.10000000149011612] (-1,-1.51) -- (0.4,-1.51) -- (0.8326237921249262,-0.1785208771867855) -- (-0.3,0.6443784760226768) -- (-1.4326237921249263,-0.17852087718678494) -- cycle;
\fill[line width=2pt,color=zzttqq,fill=zzttqq,fill opacity=0.10000000149011612] (-4.9,-1.51) -- (-3.4,-1.51) -- (-4.15,-0.21096189432334178) -- cycle;
\fill[line width=2pt,color=zzttqq,fill=zzttqq,fill opacity=0.10000000149011612] (3.08,-1.49) -- (4.4,-1.49) -- (4.4,-0.17) -- (3.08,-0.17) -- cycle;
\draw [line width=2pt,color=zzttqq] (-1,-1.51)-- (0.4,-1.51);
\draw [line width=2pt,color=zzttqq] (0.4,-1.51)-- (0.8326237921249262,-0.1785208771867855);
\draw [line width=2pt,color=zzttqq] (0.8326237921249262,-0.1785208771867855)-- (-0.3,0.6443784760226768);
\draw [line width=2pt,color=zzttqq] (-0.3,0.6443784760226768)-- (-1.4326237921249263,-0.17852087718678494);
\draw [line width=2pt,color=zzttqq] (-1.4326237921249263,-0.17852087718678494)-- (-1,-1.51);
\draw [line width=2pt,color=zzttqq] (-4.9,-1.51)-- (-3.4,-1.51);
\draw [line width=2pt,color=zzttqq] (-3.4,-1.51)-- (-4.15,-0.21096189432334178);
\draw [line width=2pt,color=zzttqq] (-4.15,-0.21096189432334178)-- (-4.9,-1.51);
\draw [line width=2pt,color=zzttqq] (3.08,-1.49)-- (4.4,-1.49);
\draw [line width=2pt,color=zzttqq] (4.4,-1.49)-- (4.4,-0.17);
\draw [line width=2pt,color=zzttqq] (4.4,-0.17)-- (3.08,-0.17);
\draw [line width=2pt,color=zzttqq] (3.08,-0.17)-- (3.08,-1.49);
\draw (-5.24,-1.61) node[anchor=north west] {1};
\draw (-3.1,-1.61) node[anchor=north west] {2};
\draw (-4.24,0.49) node[anchor=north west] {3};
\draw (-0.32,1.31) node[anchor=north west] {0};
\draw (-1.8,0.21) node[anchor=north west] {1};
\draw (-1.2,-1.49) node[anchor=north west] {2};
\draw (0.62,-1.43) node[anchor=north west] {3};
\draw (1.04,0.19) node[anchor=north west] {4};
\draw (2.82,0.35) node[anchor=north west] {0};
\draw (4.48,0.35) node[anchor=north west] {1};
\draw (4.66,-1.43) node[anchor=north west] {2};
\draw (2.78,-1.45) node[anchor=north west] {3};
\draw (-0.88,-2.77) node[anchor=north west] {Example};
\draw (2.86,-2.73) node[anchor=north west] {Non-example};
\draw (-5.2,-2.73) node[anchor=north west] {Non-example};
\draw (-0.86,-4.27) node[anchor=north west] {Figure 1};
\end{tikzpicture}

\end{figure}

Figure 1 shows some examples and non-examples of cyclically labeled polygons. The middle polygon in that figure is $P_4$.

\begin{definition}
 A \emph{polygonal subdivision} of a cyclically labeled polygon $P$ is a collection $\mathcal{P}$ of polygons contained in $P$ such that members of $\mathcal{P}$ have their vertices among the vertices of $P$, and such that any two members of $\mathcal{P}$ are either disjoint or intersect on a common edge. A 
\emph{triangulation} is a polygonal subdivision consisting of triangles. A polygonal subdivision inherits labelings on the vertices of its members from the labeling of the polygon it subdivides. A \emph{diagonal} of a polygonal subdivision is an edge of a member of the subdivision that is not an edge of the polygon being subdivided. 
\end{definition}

\definecolor{uuuuuu}{rgb}{0.26666666666666666,0.26666666666666666,0.26666666666666666}
\definecolor{zzttqq}{rgb}{0.6,0.2,0}
\begin{tikzpicture}[line cap=round,line join=round,>=triangle 45,x=1cm,y=1cm]
\fill[line width=2pt,color=zzttqq,fill=zzttqq,fill opacity=0.10000000149011612] (-1,-1.51) -- (0.4,-1.51) -- (0.8326237921249262,-0.1785208771867855) -- (-0.3,0.6443784760226768) -- (-1.4326237921249263,-0.17852087718678494) -- cycle;
\fill[line width=2pt,color=zzttqq,fill=zzttqq,fill opacity=0.10000000149011612] (-4.9,-1.51) -- (-3.4,-1.51) -- (-4.15,-0.21096189432334178) -- cycle;
\fill[line width=2pt,color=zzttqq,fill=zzttqq,fill opacity=0.10000000149011612] (3.08,-1.49) -- (4.4,-1.49) -- (4.4,-0.17) -- (3.08,-0.17) -- cycle;
\draw [line width=2pt,color=zzttqq] (-1,-1.51)-- (0.4,-1.51);
\draw [line width=2pt,color=zzttqq] (0.4,-1.51)-- (0.8326237921249262,-0.1785208771867855);
\draw [line width=2pt,color=zzttqq] (0.8326237921249262,-0.1785208771867855)-- (-0.3,0.6443784760226768);
\draw [line width=2pt,color=zzttqq] (-0.3,0.6443784760226768)-- (-1.4326237921249263,-0.17852087718678494);
\draw [line width=2pt,color=zzttqq] (-1.4326237921249263,-0.17852087718678494)-- (-1,-1.51);
\draw [line width=2pt,color=zzttqq] (-4.9,-1.51)-- (-3.4,-1.51);
\draw [line width=2pt,color=zzttqq] (-3.4,-1.51)-- (-4.15,-0.21096189432334178);
\draw [line width=2pt,color=zzttqq] (-4.15,-0.21096189432334178)-- (-4.9,-1.51);
\draw [line width=2pt,color=zzttqq] (3.08,-1.49)-- (4.4,-1.49);
\draw [line width=2pt,color=zzttqq] (4.4,-1.49)-- (4.4,-0.17);
\draw [line width=2pt,color=zzttqq] (4.4,-0.17)-- (3.08,-0.17);
\draw [line width=2pt,color=zzttqq] (3.08,-0.17)-- (3.08,-1.49);
\draw (-5.24,-1.61) node[anchor=north west] {1};
\draw (-3.1,-1.61) node[anchor=north west] {2};
\draw (-0.32,1.31) node[anchor=north west] {0};
\draw (-1.8,0.21) node[anchor=north west] {1};
\draw (-1.2,-1.49) node[anchor=north west] {2};
\draw (0.62,-1.43) node[anchor=north west] {3};
\draw (1.04,0.19) node[anchor=north west] {4};
\draw (2.82,0.35) node[anchor=north west] {0};
\draw (4.66,-1.43) node[anchor=north west] {2};
\draw (-0.88,-2.77) node[anchor=north west] {Example};
\draw (2.86,-2.73) node[anchor=north west] {Non-example};
\draw (-5.2,-2.73) node[anchor=north west] {Non-example};
\draw [line width=2pt] (-4.9,-1.51)-- (-4.14,-0.95);
\draw [line width=2pt] (-4.14,-0.95)-- (-3.4,-1.51);
\draw [line width=2pt] (3.08,-0.17)-- (4.4,-1.49);
\draw [line width=2pt] (-1.4326237921249263,-0.1785208771867848)-- (0.4,-1.51);
\draw [line width=2pt] (-1.4326237921249263,-0.1785208771867848)-- (0.8326237921249261,-0.17852087718678555);
\draw (-0.84,-4.17) node[anchor=north west] {Figure 2};
\draw (-4.22,0.35) node[anchor=north west] {0};
\draw (2.8,-1.45) node[anchor=north west] {1};
\draw (4.52,0.29) node[anchor=north west] {3};
\draw [line width=2pt] (-4.15,-0.21096189432334173)-- (-4.14,-0.95);
\draw [line width=2pt] (4.4,-0.17)-- (3.08,-1.49);
\begin{scriptsize}
\draw [fill=uuuuuu] (-4.14,-0.95) circle (2pt);
\draw [fill=uuuuuu] (-4.9,-1.51) circle (2pt);
\draw [fill=uuuuuu] (-3.4,-1.51) circle (2pt);
\draw [fill=uuuuuu] (-1.0235114100916989,-1.437639320225002) circle (2pt);
\draw [fill=uuuuuu] (-1.4326237921249263,-0.1785208771867848) circle (2pt);
\draw [fill=uuuuuu] (-0.3,0.6443784760226767) circle (2pt);
\draw [fill=uuuuuu] (0.8326237921249261,-0.17852087718678555) circle (2pt);
\draw [fill=uuuuuu] (0.4,-1.51) circle (2pt);
\draw [fill=uuuuuu] (3.08,-0.17) circle (2pt);
\draw [fill=uuuuuu] (4.4,-1.49) circle (2pt);
\draw [fill=uuuuuu] (-4.15,-0.21096189432334173) circle (2pt);
\draw [fill=uuuuuu] (4.4,-0.17) circle (2pt);
\draw [fill=uuuuuu] (3.08,-1.49) circle (2pt);
\end{scriptsize}
\end{tikzpicture}

Figure 2 above shows some examples and non-examples of polygonal subdivisions of cyclically labeled polygons. In the right picture, the triangle with vertices 0, 2, and 3 intersects the triangle with vertices 0, 1, 3, on more than a common edge. The picture on the left has subdividing polygons with a vertex not among the vertices of the polygon being subdivided. In the center picture, the edge from vertex 1 to vertex 3 is a diagonal since it is not an edge of the five-sided polygon being subdivided.

We now describe how we can define $2$-Segal maps without derived mapping objects using the middle triangulation in Figure 2 as an example. The triangulation corresponds to the category

$$\{1,2,3\}\leftarrow \{1,3\} \to \{1,3,4\} \leftarrow \{1,4\} \to \{0,1,4\}.$$

\noindent Now consider $\mathcal{P}_{\text{Fin}}(\mathbb{N})$, the category of finite ordered subsets of $(\mathbb{N}, \leq)$ and order-preserving functions, $\mathcal{P}_{\text{Fin}}(\mathbb{N})$. The diagram above can be extended to the commutative diagram \[\begin{tikzcd}[sep = small]
	{[2]} && {[1]} && {[2]} && {[1]} && {[2]} \\
	\\
	{\{1,2,3\}} && {\{1,3\}} && {\{1,3,4\}} && {\{1,4\}} && {\{0,1,4\}} \\
	\\
	{[4]} & {=} & {[4]} & {=} & {[4]} & {=} & {[4]} & {=} & {[4]}
	\arrow[hook', from=3-1, to=5-1]
	\arrow[hook', from=3-3, to=5-3]
	\arrow[hook', from=3-5, to=5-5]
	\arrow[hook', from=3-7, to=5-7]
	\arrow[hook', from=3-9, to=5-9]
	\arrow[hook', from=3-3, to=3-1]
	\arrow[hook, from=3-3, to=3-5]
	\arrow[hook', from=3-7, to=3-5]
	\arrow[hook, from=3-7, to=3-9]
	\arrow["\cong"', from=1-1, to=3-1]
	\arrow["\cong"', from=1-3, to=3-3]
	\arrow["\cong", from=1-5, to=3-5]
	\arrow["\cong", from=1-7, to=3-7]
	\arrow["\cong", from=1-9, to=3-9]
	\arrow[from=1-3, to=1-1]
	\arrow[from=1-3, to=1-5]
	\arrow[from=1-7, to=1-5]
	\arrow[from=1-7, to=1-9]
\end{tikzcd}\] in $\mathcal{P}_{\text{Fin}}(\mathbb{N})$ where the top vertical morphisms are unique isomorphisms. Deleting the middle row we get the diagram \[\begin{tikzcd}[sep=small]
	{[2]} && {[1]} && {[2]} && {[1]} && {[2]} \\
	\\
	{[4]} & {=} & {[4]} & {=} & {[4]} & {=} & {[4]} & {=} & {[4]}
	\arrow["{\{1,2,3\}}", from=1-1, to=3-1]
	\arrow["{\{1,3\}}", from=1-3, to=3-3]
	\arrow["{\{1,3,4\}}"., from=1-5, to=3-5]
	\arrow["{\{1,4\}}", from=1-7, to=3-7]
	\arrow["{\{0,1,4\}}", from=1-9, to=3-9]
	\arrow["{\delta^2_1}"', from=1-3, to=1-1]
	\arrow["{\delta^2_2}", from=1-3, to=1-5]
	\arrow["{\delta^2_2}"', from=1-7, to=1-5]
	\arrow["{\delta^2_0}", from=1-7, to=1-9]
\end{tikzcd}\] in $\Delta$ where the vertical morphisms are labeled by their images. Now if $X$ is a simplicial object in $\mathcal{C}$ we have a corresponding diagram \[\begin{tikzcd}[sep = small]
	{X_2} && {X_1} && {X_2} && {X_1} && {X_2} \\
	\\
	{X_4} & {=} & {X_4} & {=} & {X_4} & {=} & {X_4} & {=} & {X_4.}
	\arrow["{\{1,2,3\}}", from=3-1, to=1-1]
	\arrow["{\{1,3\}}", from=3-3, to=1-3]
	\arrow["{\{1,3,4\}}", from=3-5, to=1-5]
	\arrow["{\{1,4\}}", from=3-7, to=1-7]
	\arrow["{\{0,1,4\}}", from=3-9, to=1-9]
	\arrow["{d^2_1}", from=1-1, to=1-3]
	\arrow["{d^2_2}"', from=1-5, to=1-3]
	\arrow["{d^2_1}", from=1-5, to=1-7]
	\arrow["{d^2_0}"', from=1-9, to=1-7]
\end{tikzcd}\] Thus there is an induced map $X_4 \to X_2 \times_{X_1} X_2 \times_{X_1} X_2$, which we write as 
$X_4 \to X_{\{1,2,3\}} \times_{X_{\{1,3\}}} X_{\{1,3,4\}} \times_{X_{\{1,4\}}} X_{\{0,1,4\}}.$ If $\mathcal{D}$ is a sufficiently nice model category, then the induced map fits into the commutative square \[\begin{tikzcd}[sep = small]
	{\text{Map}^h(\Delta[4],X)} && {\text{Map}^h(\Delta[2]\sqcup_{\Delta[1]}\Delta[2]\sqcup_{\Delta[1]}\Delta[2],X)} \\
	\\
	{X_4} && {X_{\{1,2,3\}} \times_{X_{\{1,3\}}}^h X_{\{1,3,4\}}\times_{X_{\{1,4\}}}^h X_{\{0,1,4\}}.}
	\arrow[from=1-1, to=1-3]
	\arrow[from=3-1, to=3-3]
	\arrow["\simeq", from=1-1, to=3-1]
	\arrow["\simeq", from=1-3, to=3-3]
\end{tikzcd}\]  

We have associated a map to a given polygonal subdivision. In general the procedure is as follows. Suppose we are given a simplicial object in $\mathcal{D}$, $X$ and a polygonal subdivision $\mathcal{P}$ of an $n$-gon. First, we extend $X$  to a functor $X \colon \mathcal{P}_\text{Fin}(\mathbb{N}) \to \mathcal{D}$. After identifying objects of $\mathcal{P}_\text{Fin}(\mathbb{N})$ with order-preserving bijections of the form $f_I \colon [k] \to \{i_0,\ldots, i_k\} = I$, $X$ is given by
 $X(f_I) = X_k$ and by sending
 $\alpha\colon I = \{i_0 < i_1 < \ldots < i_k\} \to  J = \{j_0 < j_1 < \ldots < j_m\}$ to $X(f^{-1}_J  \alpha f_I)$.  And now the $2$-Segal map associated to the polygonal subdivision $\mathcal{P}$ and the functor $X$, $f_{\mathcal{P}}$, is the map from $X_n$ to the homotopy limit over $X$ of the poset category of vertex sets appearing in the subdivision. 
\begin{definition}
A simplicial object $X$ in a model category $\mathcal{C}$ is $2$-\emph{Segal} if for every polygonal subdivision $\mathcal{P}$ of $P_n$, 
the map $f_\mathcal{P}$ is a weak equivalence. The collection of maps $f_\mathcal{P}$ for $\mathcal{P}$ a polygonal subdivision of $P_n$ for some $n$ are called the $2$-\emph{Segal maps}. 
\end{definition}


\begin{proposition}
The simplicial topological space $|S_\bullet \mathcal{C}|$ is $2$-Segal for any category with cofibrations $\mathcal{C}$. 
\end{proposition}

\begin{proof}
For all $n$, $S_n\mathcal{C}$ has a zero object given by the diagram with all objects $0$. Thus $|S_n\mathcal{C}|$ is contractible. The homotopy pullback respects weak homotopy equivalence so the $2$-Segal maps of $|S_\bullet \mathcal{C}|$ are between contractible spaces. Such maps are automatically weak homotopy equivalences. 
\end{proof}

\section{Homotopy limits and projective 2-limits} \label{C4}
\thispagestyle{myheadings}

The homotopy limit of a diagram of spaces is generally more difficult to compute than the limit of a diagram of spaces. We replace homotopy limits with something a little easier to work with, what Dyckerhoff and Kapranov call projective $2$-limits  \cite[1.3.6]{DyckerhoffKapranov12}. Finally, we give a technical result on homotopy limits of diagrams of categories. 
\begin{definition}

Let $A$ be a small category and $(\mathcal{C}_a)_{a \in A}$ be a diagram of categories. The \emph{projective 2-limit} 
$2 \lim_{a \in A} \mathcal{C}_a$ is the category where an object is data consisting of:

(0) for all $a \in \text{ob}(A)$, an object $y_a$ of  $\mathcal{C}_a$; and,

(1) for all $u\colon a \rightarrow b$ in $A$, an isomorphism $y_u \colon u_* (y_a) \rightarrow y_b$ in $\mathcal{C}_b$; such that

(2) if $a\xrightarrow{u}b\xrightarrow{b}c$ is a composable pair of morphisms in $A$, then 
$y_{vu} = y_v \circ v_*(y_u)$.

\noindent A morphism in $2 \lim_{a \in A} \mathcal{C}_a$ from $(y_a,y_u)$ to $(y'_a, y'_u)$ is a system of morphisms 
$y_a \rightarrow y'_a$ in $\mathcal{C}_a$ commuting with the $y_u$ and $y'_u$. 

\end{definition}

We note the similarity between the projective $2$-limit of a diagram into $\Cat$ and a certain characterization of the ordinary limit of a diagram into $\Cat$. Let $A$ be a small category and $(\mathcal{C}_a)$ a small diagram of categories. A limit object of the diagram is given by the category where an object consists of:

(0) for all $a \in \ob(A)$, an object $y_a$ of $\mathcal{C}_a$; such that 

(1) for all $u\colon a \to b$ in $A$, $u_*(y_a) = y_b$. 

\noindent A morphism in $\lim_A \mathcal{C}_a$ from $(y_a)$ to $(y_a')$ 
is a system of morphisms $y_a \to y_a'$ such that for all $u \colon a \to b$, $u_*(y_a \to y_a') = y_b \to y_b'$. 

So the difference between the ordinary limit and the projective $2$-limit is that in the projective $2$-limit, instead of having  $u_*(y_a) = y_b$ for all $u \colon a \to b$ we replace equalities with isomorphisms and require these isomorphisms to be coherent.
 
The projective $2$-limit of a diagram of the form $A \rightarrow B \leftarrow C$ is denoted by $A \times^{(2)}_B C$ and called a \emph{2-fiber product}.

\begin{example}
Consider the category $[0]$ with one object and just the identity morphism. Let $F\colon[0] \to \Cat$ be 
a $[0]$-diagram of categories. Write $\mathcal{C}$ for the unique category in the image of the diagram. Then 
$2 \lim_{a \in [0]} \mathcal{C}_a = \mathcal{C}$ because an object of $2 \lim_{a \in [0]} \mathcal{C}_a$ is 
an object $y_a$ of $\mathcal{C}_a = \mathcal{C}$ together with an automorphism $y_u$ of $y_a$ such that $y_u = y_u^2$.
Applying $y_u^{-1}$ to both sides shows that $y_u$ must be the identity.
\end{example}

\begin{example}
We describe the $2$-fiber product $A \times_B^{(2)} C$ of the diagram $A \xrightarrow{F} B \xleftarrow{G} C$. By definition, an object is data consisting of objects $a \in \ob(A)$, $b \in \ob(B)$, and $c \in \ob(C)$, and isomorphisms $y_{id_A} \colon a \cong a$, $y_{id_B} \colon b \cong b$, $y_{id_C}\colon c \cong c$, $y_{F} \colon F(a) \cong b$, and $y_G \colon G(c) \cong b$. We write this object of $A \times_B^{(2)} C$ as the tuple 
$(a,b,c,y_{id_A}, y_{id_B}, y_{id_C}, y_F, y_G)$. Some of this data is actually superfluous. The same argument in the previous example shows that $y_{id_A}$, $y_{id_B}$, and $y_{id_C}$ are all the identity. Thus we may describe an object by only giving the data in the tuple $(a,b,c, y_{F} \colon F(a) \cong b, y_G \colon G(c) \cong b)$. A morphism from
$(a,b,c, y_{F} \colon F(a) \cong b, y_G \colon G(c) \cong b)$ to $(a',b',c', y'_{F} \colon F(a') \cong b', y'_G \colon G(c') \cong b')$ consists of morphisms $f_A\colon a \to a'$, 
$f_B \colon b \to b'$, and $f_C \colon c \to c'$ such that the diagram

\[\begin{tikzcd}[sep = small]
	{F(a)} && b && {G(c)} \\
	\\
	{F(a')} && {b'} && {G(c')}
	\arrow["{F(f_A)}"', from=1-1, to=3-1]
	\arrow["{G(f_C)}", from=1-5, to=3-5]
	\arrow[from=1-3, to=3-3]
	\arrow["\cong", from=1-1, to=1-3]
	\arrow["\cong", from=1-5, to=1-3]
	\arrow["\cong", from=3-1, to=3-3]
	\arrow["\cong", from=3-5, to=3-3]
\end{tikzcd}\] commutes. 
\end{example}

For the homotopy limits we want to replace, the model structure for $\Cat$ that we have in mind is the so-called "canonical model structure".  We begin by defining the fibrations for this model category
\begin{definition}\cite[\S 2]{Rezk00}
A functor $F \colon \mathcal{C} \to \mathcal{D}$ is an \emph{isofibration} if whenever $c$ is an object of $\mathcal{C}$ and 
$f \colon F(c) \to d$ is an isomorphism, there is a $g$ such that $F(g) = f$.
\end{definition}

\begin{proposition}\cite[3.1]{Rezk00}
There is a model category structure on $\Cat$, called the canonical model category structure, such that the fibrations are the isofibrations, the cofibrations are the functors injective on objects, and the weak equivalences are equivalences.
\end{proposition}

\begin{proposition}\cite[1.3.8]{DyckerhoffKapranov12}
Let $|-|\colon \Cat \to \Top$ denote geometric realization. For any diagram of categories $A \to \Cat$, $a \mapsto \mathcal{C}_a$, we have a natural morphism of spaces 
$f \colon |2 \lim_{a \in \text{ob}(A)} \mathcal{C}_a| \to \text{holim}_{a \in \text{ob}(A)}|\mathcal{C}_a|$. 
If $A \to \Cat$ has image in the category of groupoids, then $2\lim_{a \in \text{ob}(A)} \mathcal{C}_a$ is a groupoid and
$f$ is a weak equivalence. 
\end{proposition}

Let $C \rightarrow D \leftarrow E$ be a diagram of groupoids. In light of the previous proposition, if we had a weak equivalence between $|C \times^2_D E|$ and $|C \times^h_D E|$ then we would have a weak equivalence between $|C| \times_{|D|}^h |E|$ and $|C \times^h_D E|$. There are some strong conditions under which there is a weak equivalence between $|C \times^2_D E|$ and $|C \times^h_D E|$. We do not use these conditions in what follows but believe they are of independent interest, so we give them here. To arrive at these conditions we need to pick a particular model for the homotopy pullback of categories, which we do now. 

 There is a functorial factorization of each morphism in 
 $\Cat$ into a trivial cofibration followed by a fibration. In this factorization, the functor $F \colon C \to D$ factors through $L(C)$, a category whose objects are tuples $(c,d,\phi)$ where $c \in \text{ob}(C)$, $d \in \text{ob}(D)$, and $\phi \colon F(c) \xrightarrow{\cong} d$ is an isomorphism in $D$. In the canonical model structure, every category is fibrant, so by \cite[13.1.3]{Hirschhorn03}, $\Cat$ is right proper. Thus the homotopy pullback of $C \rightarrow D \leftarrow E$ may be defined by first factoring as $C \to L(C) \to D \leftarrow L(E) \leftarrow E$ and then taking the pullback of $L(C) \to D \leftarrow L(E)$ \cite[13.3.2]{Hirschhorn03}. We conclude that a particular model for the homotopy pullback of $C \xrightarrow{F} D \xleftarrow{K} E$ is presented by a category with objects $(c, d, e, \phi \colon F(c) \xrightarrow{\cong} d, \psi \colon F(c) = K(e) \xrightarrow{\cong} d)$, and a morphism from $(c, d, e, \phi \colon F(c) \xrightarrow{\cong} d, \psi \colon F(c) = K(e) \xrightarrow{\cong} d)$ to $(c'd',e',\phi'\colon F(c') \to d', \psi'\colon K(e') \to d')$ is given by a pair $c \to c'$ in $C$ and $e \to e'$ in $E$ such that $F(c \to c') = K(e \to e')$. 

\begin{proposition}
If $C \xrightarrow{F} D \xleftarrow{K} E$ is a a diagram of groupoids with either $F$ or $K$ full and surjective on objects, then $|C \times^h_D E|$ is weakly homotopy equivalent to 
$|C|\times^h_{|D|} |E|$. 
\end{proposition}

\begin{proof}
Without loss of generality assume $F$ is full and surjective on objects. We define a functor $\mathcal{H} \colon C \times^h_D E \to C\times^2_D E$ by sending
$(c,d,e,\phi \colon F(c) \to d, \psi \colon K(e) \to d)$ to $(c,d,e,\phi,\phi)$. Notice that $\mathcal{H}$ forgets about $\psi$. For $L\colon A \to B$ we write $L/b$ for the slice category whose objects are morphisms $L(a) \to b$. We claim that for any $(c',d', e', \phi', \psi')$ in $C\times^2_D E$, the category $\mathcal{H}/(c',d',e',\phi',\psi')$ has an initial object. Then Quillen's Theorem A \cite[1]{Quillen73} applies to the discussion after Proposition 5.6 to give the result. Let $a$ be some  object that $F$ sends to $K(e')$, $h$ be some morphism $F$ sends to $\phi'^{-1}\psi'$. We claim that for $\mathcal{H}/(c',d',e',\phi',\psi')$, an initial object is $(a, d', e', \psi', \psi')$ with the morphism part of the object given by $h, id_{d'}, id_{e'}$. To verify this data gives an initial object we take any object of $\mathcal{H}/(c',d',e',\phi',\psi')$, meaning a tuple $(c'',d'',e'',\phi'', \psi'')$ with $F(c'') = K(e'')$, $g_C \colon c'' \to c'$, $g_D \colon d'' \to d'$, and $g_E \colon e'' \to e'$ such that 

\[\begin{tikzcd}[sep = small]
	{F(c'')} && d'' && {K(e'')} \\
	\\
	{F(c')} && {d'} && {K(e')}
	\arrow["{F(g_C)}", from=1-1, to=3-1]
	\arrow["g_D", from=1-3, to=3-3]
	\arrow["{K(g_E)}", from=1-5, to=3-5]
	\arrow["{\phi''}", from=1-1, to=1-3]
	\arrow["{\phi''}", from=1-5, to=1-3]
	\arrow["{\phi'}", from=3-1, to=3-3]
	\arrow["{\psi'}", from=3-5, to=3-3]
\end{tikzcd}\] commutes.

 The unique morphism needed is given by the pair $(g_C^{-1}h,g_E^{-1})$.

\end{proof}

\section{2-Segality of discrete variants of the $S_\bullet$-construction} \label{C5}
\thispagestyle{myheadings}

Let $\mathcal{C}$ be a category with cofibrations. We determine which $2$-Segal maps for $\text{iso}(s_\bullet \mathcal{C})$ are necessarily bijections. These $2$-Segal maps turn out to be those from triangulations of $P_n$ for which each triangle contains the vertex $0$.

First, we remark that the $2$-Segal maps for $s_\bullet\mathcal{C}$ are almost never bijective. An example illustrates the general problem. Consider the $2$-Segal map $s_3\mathcal{C} \to s_{\{0,1,2\}}\mathcal{C} \times_{s_{\{0,2\}}\mathcal{C}} s_{\{0,2,3\}}\mathcal{C}$. If this map were injective then any diagram of the form \[\begin{tikzcd}[sep = small]
	&&&& {A_{2,3}} \\
	\\
	&& {A_{1,2}} && \bullet \\
	\\
	{A_{0,1}} && {A_{0,2}} && {A_{0,3}}
	\arrow[tail, from=5-1, to=5-3]
	\arrow[two heads, from=5-3, to=3-3]
	\arrow[tail, from=5-3, to=5-5]
	\arrow[curve={height=18pt}, two heads, from=5-5, to=1-5]
\end{tikzcd}\] could be uniquely completed to an object of $s_3\mathcal{C}$. Any completion of the above diagram amounts to a choice of cokernel of $A_{01} \to A_{03}$ but there is no reason a priori for there to be a unique cokernel for this morphism in $\mathcal{C}$. The general failure of any $2$-Segal map for $s_\bullet\mathcal{C}$ to be bijective is why we look at $\text{iso}(s_\bullet \mathcal{C})$ instead.  

\begin{lemma}
Suppose that for every $n \geq 3$ and $0 \leq j \leq n$, the map
$X_n \to X_{\{0, j, j+1, \ldots, n\}} \times_{X_{\{0,j\}}} X_{\{0, 1, \ldots, j\}}$ is a bijection. Then the $2$-Segal map, coming from the triangulation $\mathcal{T}$ where each triangle contains the vertex labeled $0$ under some cyclic labeling, is bijective.
\end{lemma}

\begin{proof}\label{Get0Triangle} 
We use the hypothesis repeatedly. The bijections 
\[\begin{tikzcd}[sep = small]
	{X_n} & {\xrightarrow{\cong}} & {X_{\{0, n-1, n\}} \times_{X_{\{0,n-1\}}}X_{\{0,1,\ldots, n-1\}}} \\
	{X_{\{0,1,\ldots, n-1  \}}} & {\xrightarrow{\cong}} & {X_{\{0,n-2,n-1 \}} \times_{X_{\{0,n-2  \}}} X_{\{0,1,\ldots,n-2  \}}} \\
	& \vdots \\
	{X_{\{0,1,2,3  \}}} & {\xrightarrow{\cong}} & {X_{\{0,2,3  \}} \times_{X_{\{0,2\}    }}X_{\{0,1,2  \}}}
\end{tikzcd}\]

\noindent combine to give that
$f_\mathcal{T} \colon X_n \to X_{\{0, n-1, n\}} \times_{X_{\{0,n-1\}}} X_{\{0, n - 2, n - 1\}} \times_{X_{\{0,n-2\}}} \cdots \times_{X_{\{0,2\}}} X_{\{0,1,2\}}$ is a bijection.
\end{proof}

Similarly, if  the map $X_n \to 
X_{\{0,1,\ldots,j,n\}} \times_{X_{\{j,n\}}} X_{\{j,j+1,\ldots,n\}}$ is a bijection for every $n \geq 3$ and $0 \leq j \leq n$, then the $2$-Segal map coming from the triangulation $\mathcal{T}$ where each triangle contains the vertex labeled $n$ under some cyclic labeling is bijective.

The previous lemma motivates us to define a weakened version of $2$-Segal sets.

\begin{definition}
Suppose $X$ is a simplicial object in $\mathcal{C}$.

\begin{itemize}
    \item We say $X$ is 
\emph{left $2$-Segal} if for every $n \geq 3$ and $0 \leq j \leq n$, the map
$X_n \to X_{\{0, j, j+1, \ldots, n\}} \times_{X_{\{0,j\}}}^h X_{\{0, 1, \ldots, j\}}$ is a weak equivalence. 

\item We say $X$ is \emph{right $2$-Segal} if for every $n \geq 3$ and $0 \leq j \leq n$, the map $X_n \to 
X_{\{0,1,\ldots,j,n\}} \times_{X_{\{j,n\}}}^h X_{\{j,j+1,\ldots,n\}}$ is a weak equivalence.

\end{itemize}
\end{definition}

The last fact we use to prove that certain $f_\mathcal{T}$ are always bijective is the following. 

\begin{lemma}
Let $M_n$ be the set of isomorphism classes of chains of $(n - 1)$ cofibrations. The function $\mu_n \colon \text{iso}(s_n\mathcal{C}) \longrightarrow M_n$ that sends $[A]$ to
$[A_{0,1} \rightarrowtail A_{0,2} \rightarrowtail \cdots \rightarrowtail A_{0,n}]$ is a bijection. 
\end{lemma}

\begin{proof}
An inverse is defined as follows. Suppose 
[$A_{0,1} \rightarrowtail A_{0,2} \rightarrowtail \cdots \rightarrowtail A_{0,n}]$ is an element of $M_n$. Let $A_{k,k} = 0$ for all $k$. Then for all other pairs $i,j$ with $0 < i \leq n$ and 
$0 \leq j \leq n$, let $A_{i,j}$ be a pushout object of the diagram $0 \leftarrow A_{0, i} \rightarrow A_{0,j}$ noting that all pushouts of a given diagram are isomorphic. Now, there 
are induced morphisms that make $[A]$ into an element of $\text{iso}(s_n\mathcal{C})$. 
\end{proof}

Similarly, if $\mathcal{C}$ is a category with fibrations and $E_n$ is the set of isomorphisms classes of chains of $(n - 1)$ fibrations, then the function $\text{iso}(s_n\mathcal{C}) \to E_n$ which sends $[A]$ to $[A_{0,n} \twoheadrightarrow A_{1,n} \twoheadrightarrow \cdots \twoheadrightarrow A_{n-1,n}]$ is a bijection.

\begin{proposition}
Let $\mathcal{T}$ be a triangulation of $P_n$. If each triangle of $\mathcal{T}$ contains the vertex $0$, then the $2$-Segal map $f_\mathcal{T}$ of $\text{iso}(s_\bullet \mathcal{C})$ associated to $\mathcal{T}$ is a bijection.
\end{proposition}

\begin{proof}
Our proof is an adaptation of Dyckerhoff and Kapranov's proof of a similar result \cite[2.4.8]{DyckerhoffKapranov12}. By Lemma 6.1, it suffices to show that the $2$-Segal map $\text{iso}(s_n \mathcal{C}) \rightarrow 
\text{iso}(s_{\{0,1,\ldots, j\}} \mathcal{C}) \times_{\text{iso}(s_{\{0,j\}}\mathcal{C})} \text{iso}(s_{\{0,j,j + 1, \ldots, n\}} \mathcal{C})$ is a bijection for any $j$. This $2$-Segal map fits into the commutative diagram 

\[\begin{tikzcd}[ampersand replacement=\&, sep = small]
	{\text{iso}(s_n\mathcal{C})} \&\& {\text{iso}(s_{\{0,1,\ldots,j\}}\mathcal{C})\times_{\text{iso}(s_{\{0,j\}}\mathcal{C})}\text{iso}(s_{\{0,j,j+1,\ldots,n\}}\mathcal{C})} \\
	\\
	{M_n} \&\& {M_j \times_{\text{iso}(s_{\{0,j\}}\mathcal{C})}M_{n - j + 1}}
	\arrow["{\mu_j \times \mu_{n - j + 1}}", from=1-3, to=3-3]
	\arrow[from=1-1, to=1-3]
	\arrow["{\mu_n}"', from=1-1, to=3-1]
	\arrow["{\phi_j}"', from=3-1, to=3-3]
\end{tikzcd}\] where $\phi_j$ is a bijection given by $[A_{0,1} \rightarrowtail A_{0,2} \rightarrowtail \cdots \rightarrowtail A_{0,n}] \mapsto ([A_{0,1} \rightarrowtail A_{0,2} \rightarrowtail \cdots \rightarrowtail A_{0,j}],  [A_{0,j} \rightarrowtail A_{0,j +2} \rightarrowtail \cdots \rightarrowtail A_{0,n}])$ so the result follows by Lemma 6.3.
\end{proof}

In essence, the above argument worked because having all polygons in the subdivision containing the vertex 0 means that objects in the right hand side of the $2$-Segal map have complete top rows. A complete top row is enough to specify an object in $S_\bullet \mathcal{C}$ up to isomorphism because all the lower parts of an object of $S_n\mathcal{C}$ are pushouts of the cofibrations in the top row, and taking pushouts of cofibrations is exactly what we can do in a category with cofibrations.

Call the $2$-Segal maps corresponding to triangulations where each triangle has the vertex $0$ the left $2$-Segal maps. Call the $2$-Segal maps corresponding to triangulations where each triangle has the vertex $n$ the right $2$-Segal maps. We prove that the left $2$-Segal maps are the \emph{only} $2$-Segal maps corresponding to triangulations for $\text{iso}(s_\bullet \mathcal{C})$ that are always bijective. Specifically we prove the following proposition. 

\begin{proposition}
For $\text{iso}(s_\bullet\mathcal{C})$ the $2$-Segal maps $f_\mathcal{T}$ which are not left $2$-Segal are not necessarily surjective. 
\end{proposition}

First we provide a category with cofibrations such that some of its $2$-Segal maps of the type in Proposition 6.5 are not surjective. We then use an inductive  
argument relying on the structure of triangulations of polygons to show that for this category with cofibrations none of the $2$-Segal maps of Proposition 6.5 are surjective. 

We want a category with cofibrations with as few cofibrations as possible. We are thus led to a notion of generated categories with cofibrations. 

\begin{definition}
Let $\mathcal{D}$ be a subcategory of $\mathcal{C}$ that contains a $0$ object of $\mathcal{C}$ and all morphisms between its objects and $0$ that are in $\mathcal{C}$. Then we denote the intersection of all subcategories with cofibrations of 
$\mathcal{C}$ containing $\mathcal{D}$ by $\langle \mathcal{D} \rangle$. We call $\langle \mathcal{D} \rangle$ the \emph{category with cofibrations generated by $\mathcal{D}$ in $\mathcal{C}$}.
\end{definition}

Now we justify the terminology.

\begin{lemma}
The intersection of a collection of subcategories with cofibrations of $\mathcal{C}$ is a category with cofibrations. 
\end{lemma}

\begin{proof}
Let $\mathcal{D}_a$ be a subcategory with cofibrations of $\mathcal{C}$ for all $a \in I$. Then 
$\cap_{a \in I} \mathcal{D}_a$ contains $0$ and it is a zero object for this category. Suppose $f$ is a cofibration in 
$\mathcal{C}$ which is in $\cap_{a \in I} \mathcal{D}_a$ and $g$ is a morphism in $\cap_{a \in I} \mathcal{D}_a$. Then 
the pushout of $f$ along $g$ is in all $\mathcal{D}_a$.
\end{proof}

\begin{remark}
We can also describe $\langle \mathcal{D} \rangle$ as a colimit. 
Let $\mathcal{D} = \mathcal{D}_0$ be a subcategory of $\mathcal{C}$ containing a zero object of $\mathcal{C}$ and all morphisms between its objects and $0$ that are in $\mathcal{C}$. For all $i \geq 0$, let $\mathcal{D}_{i + 1}$ be the smallest subcategory of $\mathcal{C}$ containing $\mathcal{D}_i$, all pushouts of a cofibration in $\mathcal{D}_i$ along a morphism in $\mathcal{D}_i$ and all morphisms between $0$ and objects of $\mathcal{D}$ and these pushouts. Then we claim that $\text{colim}_i \mathcal{D}_i \cong \langle \mathcal{D} \rangle$. Each $\mathcal{D}_i$ is necessarily in 
$\langle \mathcal{D} \rangle$ so we just need to show that  $\text{colim}_i \mathcal{D}_i $ is a subcategory with cofibrations
of $\mathcal{C}$. The zero object condition follows from each $\mathcal{D}_i$ satisfying the same condition. Suppose
$f$ is a morphism of $\mathcal{D}_i$ and $g$ is a cofibration of  $\mathcal{D}_j$. Then $f$ and $g$ are 
morphisms in $\mathcal{D}_{\max(i,j)}$ so the pushouts of $f$ along $g$ and $g$ along $f$ are in $\mathcal{D}_{\max(i,j) + 1}$. 
\end{remark}

The category $\mathcal{D}$ appearing in the following proposition is used to prove Proposition 6.5.

\begin{proposition}
There is a category with cofibrations $\mathcal{D}$ such that the $2$-Segal maps $\text{iso}(s_n \mathcal{D}) \to \text{iso}(s_{\{0,1,n\}}\mathcal{D}) \times_{\text{iso}(s_{\{1,n\}} \mathcal{D})} \text{iso}(s_{\{1,2,\ldots, n\}}\mathcal{D})$ are not surjective. 
\end{proposition}

\begin{proof}
We first consider the $n = 3$ case. Showing that $$\text{iso}(s_3\mathcal{D}) \to \text{iso}(s_{\{0,1,3\}}\mathcal{D}) \times_{\text{iso}(s_{\{1,3\}}\mathcal{D})} \text{iso}(s_{\{1,2,3\}}\mathcal{D})$$ is not surjective
amounts to showing that there is a diagram 

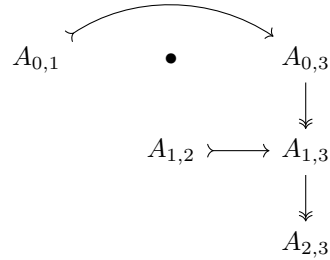
\[\begin{tikzcd}[ampersand replacement=\&, sep = small]
	{A_{0,1}} \&\& \bullet \&\& {A_{0,3}} \\
	\\
	\&\& {A_{1,2}} \&\& {A_{1,3}} \\
	\\
	\&\&\&\& {A_{2,3}}
	\arrow[curve={height=-24pt}, tail, from=1-1, to=1-5]
	\arrow[two heads, from=1-5, to=3-5]
	\arrow[tail, from=3-3, to=3-5]
	\arrow[two heads, from=3-5, to=5-5]
\end{tikzcd}\] that cannot be completed to an object of $s_3 \mathcal{D}$.

Consider the subcategory of $\mathcal{R}_f$ generated by the diagram 

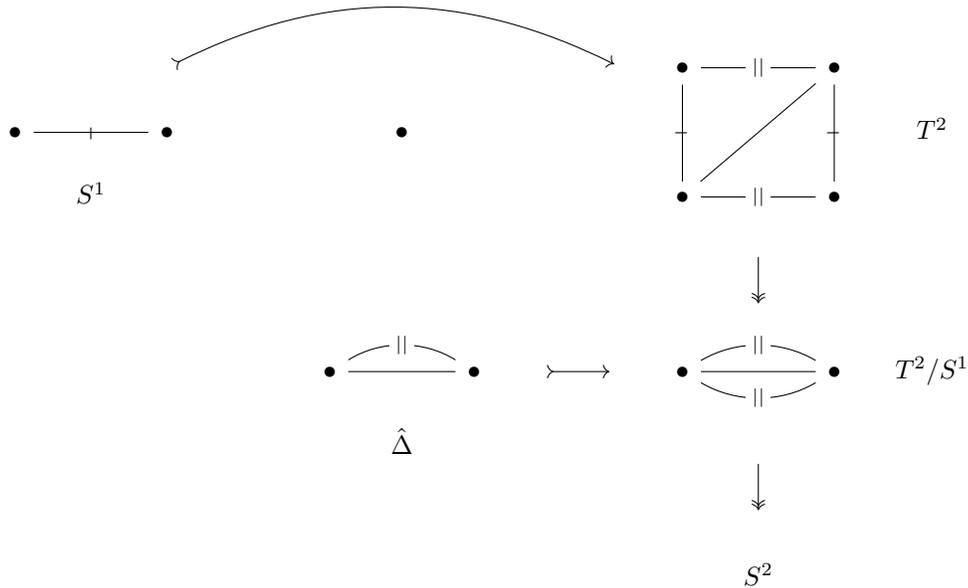
\[\begin{tikzcd}[ampersand replacement=\&, sep = small]
	\&\& {} \&\&\&\&\&\&\&\&\& {} \& \bullet \&\& \bullet \\
	\bullet  \&\& \bullet \&\&\&\&\& \bullet \&\&\&\&\&\&\&\& {T^2} \\
	\& {S^1} \&\&\&\&\&\&\&\&\&\&\& \bullet \&\& \bullet \\
	\&\&\&\&\&\&\&\&\&\&\&\&\& {} \\
	\\
	\&\&\&\&\&\&\&\&\&\&\&\&\& {} \\
	\&\&\&\&\&\& \bullet \&\& \bullet \& {} \&\& {} \& \bullet \&\& \bullet \& {T^2/S^1} \\
	\&\&\&\&\&\&\& {\hat{\Delta}} \&\&\&\&\&\& {} \\
	\\
	\&\&\&\&\&\&\&\&\&\&\&\&\& {} \\
	\&\&\&\&\&\&\&\&\&\&\&\&\& {S^2}
	\arrow["\shortmid"{marking}, no head, from=2-1, to=2-3]
	\arrow["\shortmid"{marking}, no head, from=1-13, to=3-13]
	\arrow["{||}"{description}, no head, from=1-13, to=1-15]
	\arrow["\shortmid"{marking}, no head, from=1-15, to=3-15]
	\arrow[no head, from=3-13, to=1-15]
	\arrow[curve={height=-30pt}, tail, from=1-3, to=1-12]
	\arrow["{||}"{description}, curve={height=-12pt}, no head, from=7-7, to=7-9]
	\arrow[no head, from=7-7, to=7-9]
	\arrow["{||}"{description}, no head, from=3-13, to=3-15]
	\arrow["{||}"{description}, curve={height=-12pt}, no head, from=7-13, to=7-15]
	\arrow["{||}"{description}, curve={height=12pt}, no head, from=7-13, to=7-15]
	\arrow[no head, from=7-13, to=7-15]
	\arrow[two heads, from=4-14, to=6-14]
	\arrow[tail, from=7-10, to=7-12]
	\arrow[two heads, from=8-14, to=10-14]
\end{tikzcd}\] where the same marking on different edges indicates that those edges should be identified with each other or is used to indicate the image of an object. The notation $\hat{\Delta}$ is used because the corresponding CW complex looks like a triangle turned into a hat by pinching two corners together. Let $\mathcal{D}_0$ be the smallest category containing the preceding diagram and all morphisms between its objects and $0$. To complete the above diagram to an object of $s_3\langle \mathcal{D}_0 \rangle$, we must insert 
an subcomplex of $T^2$, $X$, with $1$ $0$-cell, $3$ $1$-cells, and $1$ $2$-cell in the first row in order for $\hat{\Delta}$ to be the cokernel object of the inclusion of $S^1$ into $X$. Up to homeomorphism, the only subcomplex of $T^2$ with this number of cells in each dimension is $P$, a triangle viewed as a CW complex in the natural way with all of its vertices identified.  Recall how $\mathcal{D}_{i +1}$ (see Remark 6.8) is constructed from $\mathcal{D}_i$ and that $\langle \mathcal{D}_0 \rangle = \text{colim}_i\mathcal{D}_i$. We claim that $P$ is not in $\langle \mathcal{D}_0 \rangle$. First, note that every object in $\langle \mathcal{D}_0 \rangle$ has one $0$-cell. This is because all the objects in $\mathcal{D}_0$ have one $0$-cell as their basepoint and all the morphisms in $\mathcal{D}_0$ are basepoint preserving. Also, note that for all $j$, the objects that are in $\mathcal{D}_{j+1}$ but not $\mathcal{D}_j$ are pushout objects of pushout diagrams like

\[\begin{tikzcd}[sep = small]
	X && Y \\
	\\
	Z && {X \sqcup Y/(x \sim f(x))}
	\arrow[tail, from=1-1, to=1-3]
	\arrow["f"', from=1-1, to=3-1]
	\arrow[from=1-3, to=3-3]
	\arrow[tail, from=3-1, to=3-3]
\end{tikzcd}\] whose left vertical morphisms and top horizontal morphisms are in $\mathcal{D}_j$. This implies that if a $2$-cell in $Y$ is attached along $k$ $1$-cells then the image of that $2$-cell in $X \sqcup Y/(x \sim f(x))$ is attached along at most $k$ $1$-cells. Since the bottom map is a cellular inclusion, a $2$-cell in the pushout that is also in $Z$ is attached along the same number of $1$-cells that it is in $Z$. Since the most $1$-cells that a $2$-cell is attached along among the objects of $\mathcal{D}_0$ is $3$, it follows that all $2$-cells in any object of $\langle \mathcal{D}_0 \rangle$ are attached along at most $3$ $1$-cells. 

 Now suppose to reach a contradiction that $j$ is the smallest number for which an object of the form $W \vee P$ is in $\mathcal{D}_j$ (where $W$ might be a point). Then there is a pushout diagram of the form 

 \[\begin{tikzcd}[sep = small]
	X && Y \\
	\\
	Z && {W \vee P}
	\arrow[tail, from=1-1, to=1-3]
	\arrow["f"', from=1-1, to=3-1]
	\arrow[from=1-3, to=3-3]
	\arrow[tail, from=3-1, to=3-3]
\end{tikzcd}\] where the left vertical morphism and top horizontal morphism are in $\mathcal{D}_{j - 1}$. The $2$-cell $C$ in $P$ in $W \vee P$ is attached along $3$ $1$-cells. There is a $2$-cell in either $Z$ or $Y$ that gets mapped via the bottom horizontal or right vertical map onto $C$. Suppose a $2$-cell in $Z$ gets mapped onto $C$. Then since the bottom horizontal map is a cellular inclusion, $P$ is a wedge summand of $Z$ which is a contradiction. Suppose a $2$-cell $C'$ in $Y$ gets mapped onto $C$ by the right vertical map. Then $C'$ must be attached in $Y$ along $3$ $1$-cells by earlier remarks. If there is another $2$-cell $C''$ in $Y$ that is attached along one of the $1$-cells that $C'$ is attached along, then that common $1$-cell gets contracted by the right vertical map since there is not a $2$-cell attached to $C$ along a common $1$-cell in $W \vee P$. But then there would be at most $2$ $1$-cells that $P$ is attached along in $W \vee P$. This is a contradiction so we conclude that no $2$-cell shares a common bounding $1$-cell with $C'$. Thus $P$ is a wedge summand of $Y$, a contradiction. We conclude that $P$ is not in $\langle \mathcal{D}_0 \rangle$ so we have the $n = 3$ case of the proposition.

Now in any subcategory with cofibrations of $\mathcal{R}_f$, the $n=3$ case implies the other cases because then the diagram \[\begin{tikzcd}[ampersand replacement=\&, sep = small]
	{A_{0,1}} \&\& \bullet \&\& \bullet \&\& \cdots \&\& \bullet \&\& \bullet \&\& {A_{0,3}} \\
	\\
	\&\& 0 \&\& 0 \&\& \cdots \&\& 0 \&\& {A_{1,2}} \&\& {A_{1,3}} \\
	\&\&\&\&\&\&\&\&\&\& {} \\
	\&\&\&\& 0 \&\& \cdots \&\& 0 \&\& {A_{1,2}} \&\& {A_{1,3}} \\
	\\
	\&\&\&\&\&\& \ddots \&\& \vdots \&\& \vdots \&\& \vdots \\
	\\
	\&\&\&\&\&\&\&\& 0 \&\& {A_{1,2}} \&\& {A_{1,3}} \\
	\\
	\&\&\&\&\&\&\&\&\&\& {A_{1,2}} \&\& {A_{1,3}} \\
	\\
	\&\&\&\&\&\&\&\&\&\&\&\& {A_{2,3}}
	\arrow[curve={height=-24pt}, tail, from=1-1, to=1-13]
	\arrow[tail, from=3-3, to=3-5]
	\arrow[tail, from=3-5, to=3-7]
	\arrow[tail, from=3-7, to=3-9]
	\arrow[tail, from=3-9, to=3-11]
	\arrow[tail, from=3-11, to=3-13]
	\arrow[two heads, from=1-13, to=3-13]
	\arrow[two heads, from=3-5, to=5-5]
	\arrow[two heads, from=3-9, to=5-9]
	\arrow[two heads, from=3-11, to=5-11]
	\arrow[two heads, from=3-13, to=5-13]
	\arrow[tail, from=5-5, to=5-7]
	\arrow[tail, from=5-7, to=5-9]
	\arrow[tail, from=5-9, to=5-11]
	\arrow[tail, from=5-11, to=5-13]
	\arrow[two heads, from=5-9, to=7-9]
	\arrow[two heads, from=5-11, to=7-11]
	\arrow[two heads, from=5-13, to=7-13]
	\arrow[two heads, from=7-9, to=9-9]
	\arrow[two heads, from=7-11, to=9-11]
	\arrow[two heads, from=7-13, to=9-13]
	\arrow[tail, from=9-9, to=9-11]
	\arrow[tail, from=9-11, to=9-13]
	\arrow[tail, from=11-11, to=11-13]
	\arrow[two heads, from=9-11, to=11-11]
	\arrow[two heads, from=9-13, to=11-13]
	\arrow[two heads, from=11-13, to=13-13]
\end{tikzcd}\]

 cannot be completed to an object of $s_n\mathcal{D}$. 

\end{proof}

\begin{remark}
In a conversation, Maxine Calle pointed out that for any finite CW complex $X$ we can do the above argument in an analogous subcategory of $\mathcal{R}_f(X)$ since the $X$ gets carried along in taking pushouts. 
\end{remark}

We now take a detour into the structure of triangulations of polygons needed to set up an inductive argument for the proof of Proposition 6.5. 

\begin{lemma}\label{StrictConsecTri}
Every triangulation of a $P_n$ has a triangle whose vertices are labeled by
$j - 1, j, j+ 1$ for some $j$.
\end{lemma}

\begin{proof}
The vertices of $P_2$ are labeled by $0,1,2$ so we can take $j = 1$. There are two triangulations of $P_3$. The triangulation with a diagonal from $0$ to $2$ contains the triangle labeled by $0$,$1$,$2$. The triangulation with a diagonal from $1$ to $3$ contains the triangle labeled by $1$,$2$,$3$. We proceed by induction. Suppose the result is true for $n = k$. Consider a triangulation of the regular $(k + 1)$-gon.
Let $i < j$ be the labels for the vertices of a diagonal of the triangulation. If $j = i + 2$ then we are done. If $j \neq i + 2$ 
then consider the induced triangulation on the polygon with vertices $i, i +1, \ldots j$. By inductive assumption, this triangulation contains
a triangle of the desired form which is then in the larger triangulation. The result follows by induction.
\end{proof}

Note in the previous lemma that $j - 1, j , j + 1$ are strictly consecutive, not just consecutive modulo $n$. So what is the difference in significance between triangles with consecutive vertices and triangles with vertices consecutive modulo $n$? Consider the following example. The triangulation

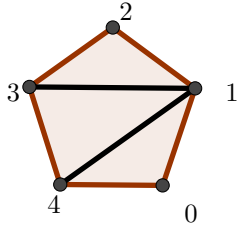
\definecolor{zzttqq}{rgb}{0.6,0.2,0}
\definecolor{uuuuuu}{rgb}{0.26666666666666666,0.26666666666666666,0.26666666666666666}
\begin{tikzpicture}[line cap=round,line join=round,>=triangle 45,x=.3cm,y=.3cm]

\fill[line width=2pt,color=zzttqq,fill=zzttqq,fill opacity=0.10000000149011612] (-2.9155046438166723,-0.0887557150591822) -- (1.6217053587551462,-0.12820971508154583) -- (3.0613033404130148,4.174741467160951) -- (-0.5861861793584916,6.873565549740722) -- (-4.280056657844184,4.238579380189277) -- cycle;
\draw [line width=2pt,color=zzttqq] (-2.9155046438166723,-0.0887557150591822)-- (1.6217053587551462,-0.12820971508154583);
\draw [line width=2pt,color=zzttqq] (1.6217053587551462,-0.12820971508154583)-- (3.0613033404130148,4.174741467160951);
\draw [line width=2pt,color=zzttqq] (3.0613033404130148,4.174741467160951)-- (-0.5861861793584916,6.873565549740722);
\draw [line width=2pt,color=zzttqq] (-0.5861861793584916,6.873565549740722)-- (-4.280056657844184,4.238579380189277);
\draw [line width=2pt,color=zzttqq] (-4.280056657844184,4.238579380189277)-- (-2.9155046438166723,-0.0887557150591822);
\draw [line width=2pt] (-4.280056657844184,4.238579380189277)-- (3.0613033404130148,4.174741467160951);
\draw [line width=2pt] (-2.9155046438166723,-0.0887557150591822)-- (3.0613033404130148,4.174741467160951);
\draw (2.1740613590682396,-0.6016577153499093) node[anchor=north west] {0};
\draw (3.988945360096967,4.764086287691543) node[anchor=north west] {1};
\draw (-0.7060806425643064,8.354400289726632) node[anchor=north west] {2};
\draw (-5.70110664522558,4.764086287691543) node[anchor=north west] {3};
\draw (-3.9018546443757613,-0.20711771512627306) node[anchor=north west] {4};
\begin{scriptsize}
\draw [fill=uuuuuu] (-2.9155046438166723,-0.0887557150591822) circle (2.5pt);
\draw [fill=uuuuuu] (1.6217053587551462,-0.12820971508154583) circle (2.5pt);
\draw [fill=uuuuuu] (3.0613033404130148,4.174741467160951) circle (2.5pt);
\draw [fill=uuuuuu] (-0.5861861793584916,6.873565549740722) circle (2.5pt);
\draw [fill=uuuuuu] (-4.280056657844184,4.238579380189277) circle (2.5pt);
\end{scriptsize}
\end{tikzpicture}

\noindent can be built up as 

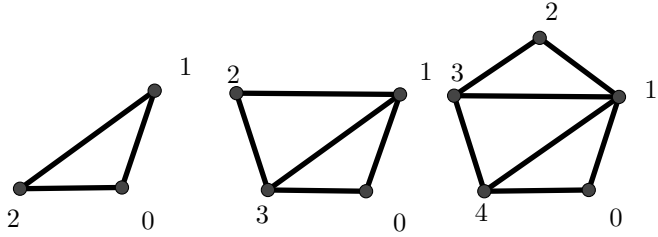
\begin{tikzpicture}[line cap=round,line join=round,>=triangle 45,x=.3cm,y=.3cm]

\draw [line width=2pt] (-13.798111398031187,-10.136960699396816)-- (-9.281757040350023,-10.081203238190877);
\draw [line width=2pt] (-7.832063048995575,-5.787878725333492)-- (-9.281757040350023,-10.081203238190877);
\draw [line width=2pt] (-13.798111398031187,-10.136960699396816)-- (-7.832063048995575,-5.787878725333492);
\draw [line width=2pt] (-4.207828070609456,-5.899393647745372)-- (-2.8138915404609492,-10.192718160602757);
\draw [line width=2pt] (-2.8138915404609492,-10.192718160602757)-- (1.5351904336023938,-10.248475621808698);
\draw [line width=2pt] (1.5351904336023938,-10.248475621808698)-- (3.0406418861627817,-5.955151108951312);
\draw [line width=2pt] (-4.207828070609456,-5.899393647745372)-- (3.0406418861627817,-5.955151108951312);
\draw [line width=2pt] (-2.8138915404609492,-10.192718160602757)-- (3.0406418861627817,-5.955151108951312);
\draw [line width=2pt] (5.438212718018215,-6.010908570157252)-- (6.776391786960781,-10.248475621808698);
\draw [line width=2pt] (9.229720080022155,-3.446065354684009)-- (5.438212718018215,-6.010908570157252);
\draw [line width=2pt] (9.229720080022155,-3.446065354684009)-- (12.742440135996393,-6.066666031363193);
\draw [line width=2pt] (6.776391786960781,-10.248475621808698)-- (11.404261067053826,-10.192718160602757);
\draw [line width=2pt] (12.742440135996393,-6.066666031363193)-- (11.404261067053826,-10.192718160602757);
\draw (-8.835697350702501,-10.750292772662158) node[anchor=north west] {0};
\draw (-7.162973514524292,-3.94788250553747) node[anchor=north west] {1};
\draw (-14.801745699738113,-10.69453531145622) node[anchor=north west] {2};
\draw (2.315794890485558,-10.86180769507404) node[anchor=north west] {0};
\draw (3.486701575810304,-4.17091235036123) node[anchor=north west] {1};
\draw (-5.044189988698561,-4.282427272773111) node[anchor=north west] {2};
\draw (-3.761768380961934,-10.527262927838398) node[anchor=north west] {3};
\draw [line width=2pt] (5.438212718018215,-6.010908570157252)-- (12.742440135996393,-6.066666031363193);
\draw [line width=2pt] (6.776391786960781,-10.248475621808698)-- (12.742440135996393,-6.066666031363193);
\draw (11.906078217907288,-10.638777850250278) node[anchor=north west] {0};
\draw (13.467287131673617,-4.951516807244391) node[anchor=north west] {1};
\draw (9.062447696404334,-1.494554212476107) node[anchor=north west] {2};
\draw (4.880638105958812,-4.22666981156717) node[anchor=north west] {3};
\draw (5.940029868871677,-10.415748005426519) node[anchor=north west] {4};
\begin{scriptsize}
\draw [fill=uuuuuu] (-13.798111398031187,-10.136960699396816) circle (2.5pt);
\draw [fill=uuuuuu] (5.438212718018215,-6.010908570157252) circle (2.5pt);
\draw [fill=uuuuuu] (-9.281757040350023,-10.081203238190877) circle (2.5pt);
\draw [fill=uuuuuu] (-7.832063048995575,-5.787878725333492) circle (2.5pt);
\draw [fill=uuuuuu] (-4.207828070609456,-5.899393647745372) circle (2.5pt);
\draw [fill=uuuuuu] (-2.8138915404609492,-10.192718160602757) circle (2.5pt);
\draw [fill=uuuuuu] (1.5351904336023938,-10.248475621808698) circle (2.5pt);
\draw [fill=uuuuuu] (3.0406418861627817,-5.955151108951312) circle (2.5pt);
\draw [fill=uuuuuu] (6.776391786960781,-10.248475621808698) circle (2.5pt);
\draw [fill=uuuuuu] (11.404261067053826,-10.192718160602757) circle (2.5pt);
\draw [fill=uuuuuu] (12.742440135996393,-6.066666031363193) circle (2.5pt);
\draw [fill=uuuuuu] (9.229720080022155,-3.446065354684009) circle (2.5pt);
\end{scriptsize}
\end{tikzpicture}

Notice that at each step the last triangle added had its vertices consecutive modulo $n$ and that the last triangle added always has a vertex not connected to any diagonals. 

\begin{lemma}\label{ConsecTri}
Viewing the edges and vertices of a triangulation of a $P_n$ as a graph, every vertex has valency at least $2$. The vertices of valency $2$ are 
exactly those $j$ for which $j - 1, j, j + 1$ mod $n$ is a triangle in the triangulation.
\end{lemma}

\begin{proof}
Every triangulation can be constructed by attaching triangles to each other along edges like the example above. Starting with a triangle with vertices cyclically labeled, we add a triangle glued along one of the edges of the original triangle. Then we relabel the vertices keeping $0$ at the same place. Now we add another triangle and relabeling while keeping 0 at the same place and repeat. The construction of the triangulation of the pentagon above gives an example. At the end of each relabeling the last triangle added has become one with vertices labeled $j - 1, j, j + 1$ mod $k$ for some $j$ and $k$. The process of attaching a triangle as above increases the degree of two of the vertices by $1$ and adds a vertex of degree $2$; namely, the middle vertex in a triangle with consecutive vertices modulo $n$. 
\end{proof}

\begin{lemma}
When $\mathcal{T}$ is a triangulation containing the triangle $T_0$ which we identify with the label set of its vertices $\{j-1,j,j+1\}$, the natural morphisms of the form $\lim_\mathcal{T}\text{iso}(s_\bullet \mathcal{C}) \to \lim_{\mathcal{T}, T_0}\text{iso}(s_\bullet \mathcal{C})$ are surjective. Here, $\lim_{\mathcal{T}, T_0}\text{iso}$ is abbreviated notation for $\lim_{\mathcal{T} - \{T_0,\{j-1,j+1\}\}}$.
\end{lemma}

\begin{proof}
Suppose the triangles of $\mathcal{T}$ are $T_1, T_2, \ldots, T_{n - 2},$ and $ T_0:= \{j - 1,j, j+1\}$ where we identify each triangle with the set of lables on its vertices. Then $\lim_\mathcal{T}\text{iso}(s_\bullet \mathcal{C})$ may be represented by the set whose elements are tuples of isomorphism classes $([A_1], [A_2], \ldots, [A_{n - 2}], [A_0])$ where 
$A_i \in s_{T_i}\mathcal{C}$ for all $i$, and if $T_i \cap T_k = \{\ell, m\}$, then the images of $[A_i]$ and $[A_k]$ in $\text{iso}(s_{\{\ell,m\}}\mathcal{C})$ are the same. 
Similarly, $\lim_{\mathcal{T}, T_0} \text{iso}(s_\bullet \mathcal{C})$ may be represented by the set of tuples of isomorphism classes $([A_1], [A_2], \ldots, [A_{n - 2}])$ where 
$A_i \in s_{T_i}\mathcal{C}$ for all $i$, and if $T_i \cap T_k = \{\ell, m\}$, then the images of $[A_i]$ and $[A_k]$ in $\text{iso}(s_{\{\ell,m\}})$ are the same. Consider a generic element $([A_1], [A_2], \ldots, [A_{n - 2}])$ of 
$\lim_{\mathcal{T}, T_0} \text{iso}(s_\bullet \mathcal{C})$. There is a unique $k$ and $\alpha$ such that 
$\{k,j-1,j+1\} = T_\alpha$ is the only other triangle of $\mathcal{T}$ that shares an edge with $\{j - 1, j, j+1\}$. Assume $k < j - 1$. The other case ($k > j + 1$) is similar. 
A preimage of this generic element is given by 
$([A_1], [A_2],\ldots, [A_{n-2}], [0 \rightarrowtail (A_\alpha)_{j - 1, j + 1} \twoheadrightarrow (A_\alpha)_{j-1,j+1}])$. 
\end{proof}

\begin{proof}[Proof of Proposition 6.5]
We proceed by induction on the number of sides of the polygon being
subdivided. For $n =3$ the result follows from Proposition 6.9. Let $\mathcal{C}$ be the subcategory of $\mathcal{R}_f$ from the proof of Proposition 6.9. Suppose the result is true for $n = k$. Let
 $\mathcal{T}$ be a triangulation of the regular $(k + 1)$-gon with cyclically labeled vertices where some triangle in the triangulation does not contain the $0$ vertex. We view the vertices and edges of the triangulation as a graph. There are two cases: either the degree 
of the vertex $0$ is $2$ or it is greater than $2$. 

Suppose the degree of the vertex $0$ is $2$. Then by Lemma \ref{ConsecTri}, the triangle given by $0,1,k + 1$ is in $\mathcal{T}$. The induced triangulation on the polygon with
vertices labeled from $1$ to $k + 1$ contains a triangle with vertices of the form $j-1,j, j + 1$ by Lemma \ref{StrictConsecTri}.
We have a commutative diagram 

\[\begin{tikzcd} [sep = small]
	{\text{iso}(s_{k + 1}\mathcal{C})} && {\text{lim}_\mathcal{T} \text{iso}(s_\bullet\mathcal{C})} \\
	\\
	{\text{iso}(s_{\{0,1,\ldots, j-1, j+1, \ldots k + 1\}}\mathcal{C})} && {\text{lim}_{\mathcal{T} - \{j,j-1,j+1\} - \{j,j+1\}} \text{iso}(s_\bullet \mathcal{C})}
	\arrow[from=1-1, to=1-3]
	\arrow[from=1-1, to=3-1]
	\arrow[from=3-1, to=3-3]
	\arrow[from=1-3, to=3-3]
\end{tikzcd}\] where by $\{j,j-1,j+1\}$ we mean the triangle with those vertices. The top and bottom maps 
are $2$-Segal maps, the left vertical map is the appropriate face map of $\text{iso}(s_\bullet \mathcal{C})$ which is surjective, and the right vertical 
map is projection. The right vertical map is surjective by Lemma 6.13. The inductive hypothesis gives that the bottom horizontal map is not surjective so the top horizontal map cannot be surjective. 

Now suppose the degree of the vertex labeled $0$ is greater than $2$. Then there is a $j \neq 1,k + 1$ such that the edge from $0$ to $j$ is a diagonal of the triangulation. This diagonal splits $P_n$ into two cyclically labeled triangulated polygons. One of these two triangulated polygons must now contain a triangle that does not contain the vertex labeled $0$. Proceeding like above we can reduce our attention to the $2$-Segal map corresponding to the triangulated sub-polygon with a triangle not containing $0$. Look at the degree of the vertex $0$; if it is not $2$ repeat this procedure until the degree of the vertex $0$ is $2$. Then proceed as in the first case. 
\end{proof}

We have shown that for all categories with cofibrations $\mathcal{C}$, $\text{iso}(s_\bullet \mathcal{C})$ is left $2$-Segal but need not be fully $2$-Segal.

\section{2-Segality of topological versions of the $S_\bullet$-construction} \label{C6}
\thispagestyle{myheadings}

Let $\mathcal{C}$ be a category with cofibrations. The goals of this chapter are threefold. First, we show that the left $2$-Segal maps for $|iS_\bullet \mathcal{C}|$ are weak equivalences. Second, we show these are the only $2$-Segal maps, corresponding to triangulations, which are necessarily weak equivalences. The proofs given are more categorical versions of the proofs given for $\text{iso}(s_\bullet \mathcal{C})$. Finally, we give a partial result for $|wS_\bullet \mathcal{C}|$. 

\begin{proposition}
Let $\mathcal{C}$ be a category with cofibrations and $\mathcal{T}$ be the triangulation of $P_n$ where each triangle of $\mathcal{T}$ contains the vertex $0$. Then $f_\mathcal{T}$ is a weak equivalence.
\end{proposition}

\begin{lemma}
Let $M_n\mathcal{C}$ be the groupoid whose objects are chains of $(n - 1)$ cofibrations in $\mathcal{C}$ and whose morphisms are isomorphisms of such chains. The functor $\mu_n \colon iS_n \mathcal{C} \rightarrow M_n\mathcal{C}$ which sends $A$ to
$A_{0,1} \rightarrowtail A_{0,2} \rightarrowtail \cdots \rightarrowtail A_{0,n}$ is an equivalence. 
\end{lemma}

The proof is analogous to that of Lemma 6.3.

\begin{proof}[Proof of Proposition 7.1] By Proposition 5.6 it suffices to show that the functor $\Phi_j \colon iS_n \mathcal{C} \rightarrow 
iS_{\{0,1,\ldots, j\}} \mathcal{C} \times^{(2)}_{S_{\{0,j\}}\mathcal{C}} iS_{\{j,j + 1, \ldots, n\}} \mathcal{C}$ is an equivalence. 
There is a commutative diagram 

\[\begin{tikzcd}[ampersand replacement=\&, sep = small]
	{iS_n \mathcal{C}} \&\& {iS_{\{0,1,\ldots, j\}}\mathcal{C} \times^{(2)}_{iS_{\{0,j\}}\mathcal{C}} iS_{\{j, j + 1, \ldots, n\}} \mathcal{C}} \\
	\\
	M_n \&\& M_j \times^{(2)}_{iS_{\{0,j\}}\mathcal{C}} M_{n - j + 1}
	\arrow["{\mu_n}"', from=1-1, to=3-1]
	\arrow["{\Phi_j}", from=1-1, to=1-3]
	\arrow["{\mu_j \times \mu_{n - j + 1}}", from=1-3, to=3-3]
	\arrow["{\phi_j}"', from=3-1, to=3-3]
\end{tikzcd}\] where $\phi_j$ is an equivalence so $\Phi_j$ is an equivalence. 
\end{proof}

Now we turn to the "only" part of the result. 

\begin{lemma}\label{EssentSurj}
A functor between groupoids that is not essentially surjective cannot induce a weak homotopy equivalence between realizations.
\end{lemma}

\begin{proof}
Suppose $F \colon \mathcal{C} \rightarrow \mathcal{D}$ is a functor between groupoids that induces a weak homotopy equivalence $|\mathcal{C}| \rightarrow |\mathcal{D}|$. Then $|F|$ gives a surjection (in fact a bijection) between path components of the $1$-skeleta. Since all edges in the $1$-skeleta come from isomorphisms this cannot be the case unless $F$ is essentially surjective.
\end{proof}

Because of Lemma \ref{EssentSurj}, it suffices to show that the maps from $iS_n\mathcal{C}$ to $2$-fiber products of categories corresponding to $2$-Segal maps are not necessarily essentially surjective.

\begin{lemma}\label{Projection}
 When $\mathcal{T}$ is a triangulation containing the triangle $\{j-1,j,j+1\}$, the natural morphisms of the form $2\lim_\mathcal{T}(iS_\bullet \mathcal{C}) \to 2\lim_{\mathcal{T} - \{j-1,j,j+1\} - \{j-1,j+1\}}(iS_\bullet \mathcal{C})$ are surjective on objects.
\end{lemma}

The proof is the same as for Lemma 6.13 with "bijection" replaced by "essentially surjective" except that one must keep track of the isomorphisms that are part of the data of an object in a $2$ limit. We use the fact that a composite of essentially surjective functors is essentially surjective and if $G$ and $F$ are functors and $GF$ is essentially surjective, then so is $G$. 

\begin{proposition}\label{BadSubdivision}
There is a category with cofibrations $\mathcal{D}$ such that the natural map 
$iS_n \mathcal{D} \rightarrow iS_{\{0,1,n\}} \mathcal{D} \times^{(2)}_{iS_{\{1,n\}} \mathcal{D}} iS_{\{1,2,\ldots,n\}}\mathcal{D}$ is not essentially surjective. 
\end{proposition}

The proof is essentially the same as the proof of Proposition 6.9. 

\begin{proposition}
The $2$-Segal maps of $|iS_\bullet\mathcal{C}|$ which are not left $2$-Segal are not necessarily weak equivalences. 
\end{proposition}

\begin{proof}
By Lemma \ref{EssentSurj} and Proposition 5.6, it suffices to show that the corresponding maps into the projective $2$-limits of categories are not necessarily essentially surjective. The proof is the same as the proof of Proposition 6.5, but with limits replaced by projective $2$-limits. 

\end{proof}

Now having the weak equivalences all be isomorphisms is an unusually special situation. Another unusually special situation is having all morphisms be weak equivalences. We have seen that $|S_\bullet \mathcal{C}|$ is $2$-Segal and $|iS_\bullet\mathcal{C}|$ is left $2$-Segal and may be $2$-Segal. What can we say in the middle cases of arbitrary Waldhausen categories? So far we have not been able to say anything, but if $\mathcal{C}$ is a nice Waldhausen category we can say the following. 

\begin{proposition}
If $\mathcal{C}$ is a saturated Waldhausen category satisfying the mapping cylinder axiom, then a $2$-Segal map of $|wS_\bullet \mathcal{C}|$ is a weak equivalence if and only if the corresponding $2$-Segal map for $|wcofM_\bullet \mathcal{C}|$ is a weak equivalence. 
\end{proposition}

\begin{proof}
By Proposition 3.15, the maps $wcofS_n \mathcal{C} \to S_n \mathcal{C}$ are homotopy equivalences. Analogously to Lemma 7.2, the maps $wcofM_n\mathcal{C} \to wcofS_n\mathcal{C}$ are homotopy equivalences. All of these maps commute with face and degeneracy maps. Since in $\Top$ all spaces are fibrant, the result follows. 
\end{proof}

\section{2-Segality of categorical versions of the $S_\bullet$-construction} \label{C7}
\thispagestyle{myheadings}

We have seen that $|S_\bullet \mathcal{C}|$ is always $2$-Segal (Proposition 4.16). Here we show that $S_\bullet \mathcal{C}$ is always left $2$-Segal for $\mathcal{C}$ a category with cofibrations but is not always $2$-Segal. In particular, a simplicial object in $\Cat$ can be left but not fully $2$-Segal while its geometric realization is fully $2$-Segal. We will also see that $wS_\bullet\mathcal{C}$ behaves similarly to $S_\bullet \mathcal{C}$. Some of our arguments are parallel to the ones given in the previous two sections. 

\begin{lemma}
Let $M_n$ be the category whose objects are chains of $(n - 1)$ cofibrations in $\mathcal{C}$ and whose morphisms are morphisms of such chains. The functor $\mu_n \colon S_n \mathcal{C} \to M_n$ which sends $A$ to 
$A_{0,1} \rightarrowtail A_{0,2} \rightarrowtail \cdots \rightarrowtail A_{0,n}$ is an equivalence. 
\end{lemma}

\begin{lemma}
If $i \colon \{0,j\} \to \alpha = \{\alpha_1, \ldots, \alpha_l \}$ is injective then $S_i\mathcal{C} \colon S_\alpha\mathcal{C} \to S_{\{0,j\}} \mathcal{C}$ is an isofibration. 
\end{lemma}

\begin{proof}
Let $A$ be an object of $S_\alpha \mathcal{C}$ and $\tilde{f} \colon 
(S_i\mathcal{C})(A) = A_{0,j} \to B_{0,j}$ be an isomorphism in $S_{\{0,j\}}\mathcal{C}$. We want to
produce an object $B$ of $S_\alpha \mathcal{C}$ and a morphism $f \colon A \to B$ such that $(S_i \mathcal{C})(f) = \tilde{f}$. We gain both $B$ and $f$ by extending the diagram 

\[\begin{tikzcd} [row sep=3mm,column sep = 3mm ,ampersand replacement=\&]
	{A_{0,\alpha_1}} \&\& {A_{0,\alpha_2}} \&\& \cdots \&\& {A_{0,\alpha_k}} \&\& {B_{0,j}} \&\& {A_{0,\alpha_{k + 1}}} \&\& \cdots \&\& {A_{0,\alpha_l}} \\
	\\
	{A_{0,\alpha_1}} \&\& {A_{0,\alpha_2}} \&\& \cdots \&\& {A_{0, \alpha_{k}}} \&\& {A_{0,j}} \&\& {A_{0,\alpha_{k + 1}}} \&\& \cdots \&\& {A_{0, \alpha_l}}
	\arrow["{a_2}", tail, from=3-1, to=3-3]
	\arrow["{a_3}", tail, from=3-3, to=3-5]
	\arrow["{a_k}", tail, from=3-5, to=3-7]
	\arrow["a", tail, from=3-7, to=3-9]
	\arrow["{a_{k +1}}", tail, from=3-9, to=3-11]
	\arrow["{a_{k + 2}}", tail, from=3-11, to=3-13]
	\arrow["{a_l}", tail, from=3-13, to=3-15]
	\arrow["{=}"', from=3-1, to=1-1]
	\arrow["{=}"', from=3-3, to=1-3]
	\arrow["{=}"', from=3-7, to=1-7]
	\arrow["{a_2}", tail, from=1-1, to=1-3]
	\arrow["{a_3}", tail, from=1-3, to=1-5]
	\arrow["{a_k}", tail, from=1-5, to=1-7]
	\arrow["{\tilde{f}}"{description}, from=3-9, to=1-9]
	\arrow["{\tilde{f}a}", tail, from=1-7, to=1-9]
	\arrow["{a_{k + 1}\tilde{f}^{-1}}", tail, from=1-9, to=1-11]
	\arrow["{=}", from=3-11, to=1-11]
	\arrow["{a_{k + 2}}", tail, from=1-11, to=1-13]
	\arrow["{a_l}", tail, from=1-13, to=1-15]
	\arrow["{=}", from=3-15, to=1-15]
\end{tikzcd}\] using the equivalence of the previous lemma.

\end{proof}

We conclude that for any $j$, $S_{\{0,1,\ldots, j\}}\mathcal{C} \times_{S_{\{0,j\}}\mathcal{C}}^h S_{\{0,j,j + 1,\ldots, n\}} \mathcal{C}$ is equivalent to the ordinary pullback. 
Now in the same way as Lemma 6.3 and Proposition 7.2, there is a commutative diagram
\[\begin{tikzcd}[ampersand replacement=\&]
	{S_n \mathcal{C}} \&\&\& {S_{\{0,1,\ldots, j\}}\mathcal{C}\times_{S_{\{0,j\}}\mathcal{C}}S_{\{0,j,j+1,\ldots,n\}}\mathcal{C}} \\
	\\
	{M_n} \&\&\& {M_j \times_{S_{\{0,j\}}\mathcal{C}}M_{n - j + 1}.}
	\arrow[from=1-1, to=1-4]
	\arrow[from=1-1, to=3-1]
	\arrow[from=3-1, to=3-4]
	\arrow[from=1-4, to=3-4]
\end{tikzcd}\] The bottom and vertical maps are equivalences, so the top map is an equivalence. 

Now suppose $\mathcal{C}$ is a Waldhausen category. Consider the map $wS_n  \mathcal{C} \to wS_{\{0,1,\ldots,j\}}\mathcal{C}$
$ \times_{wS_{\{0,j\}}\mathcal{C}}^h wS_{\{0,j,j+1,\ldots,n\}}\mathcal{C}$. The isomorphisms in $wS_m\mathcal{C}$ are the same as the isomorphisms in $S_m\mathcal{C}$. Thus 
$wS_{\{0,1,\ldots,j\}}\mathcal{C} \to wS_{\{0,j\}}\mathcal{C}$ is an isofibration and the homotopy pullback can be replaced by an ordinary pullback. We can for each $n$, consider the category $wM_n$ whose objects are chains of $n - 1$ cofibrations in $\mathcal{C}$ and whose morphisms are componentwise weak equivalences. We get, like above for $S_\bullet \mathcal{C}$, a commutative diagram 

\[\begin{tikzcd}[sep = small]
	{wS_n\mathcal{C}} && {wS_{\{0,1,\ldots, j\}}\mathcal{C}\times_{wS_{\{0,j\}}\mathcal{C}}wS_{\{0,j,j+1,\ldots,n\}}\mathcal{C}} \\
	\\
	{wM_n} && {wM_j \times_{wS_{\{0,j\}}\mathcal{C}}wM_{n - j + 1}}
	\arrow[from=1-1, to=1-3]
	\arrow[from=1-1, to=3-1]
	\arrow[from=1-3, to=3-3]
	\arrow[from=3-1, to=3-3]
\end{tikzcd}\] where the bottom and vertical maps are equivalences..

From the above considerations in $S_\bullet \mathcal{C}$ and $wS_\bullet\mathcal{C}$ we get the following proposition. 

\begin{proposition}
The left $2$-Segal maps of $S_\bullet\mathcal{C}$ and $wS_\bullet \mathcal{C}$, are equivalences of categories. 
\end{proposition}

\begin{remark}
The nerve functor applied to a morphism which is the image of an injection under $S_\bullet\mathcal{C}$ is not a fibration of simplicial sets since this nerve is not cofibered over groupoids \cite[2.1.1.2]{Lurie09}. Since $N(S_i\mathcal{C})$ is not a fibration we do not have a reason to be able to replace a homotopy pullback with a pullback when studying the $2$-Segal maps of $|S_\bullet \mathcal{C}|$. Just like the $S_\bullet \mathcal{C}$ case,  $wS_{\{0,1,\ldots,j\}}\mathcal{C} \to wS_{\{0,j\}}\mathcal{C}$ is never fibered in groupoids.
\end{remark}

\begin{proposition}
There is a category with cofibrations $\mathcal{D}$ such that the natural map 
$S_n \mathcal{D} \to S_{\{0,1,n\}}\mathcal{D} \times_{S_{\{1,n\}}\mathcal{D}}S_{\{1,2,\ldots,n\}}\mathcal{D}$ is not essentially surjective.
\end{proposition}

The same category $\mathcal{D}$ used for Proposition 6.9 can be used here, and the proof is the same as the one there. Since equivalences of categories must be essentially surjective, we get that the $2$-Segal map $S_n \mathcal{D} \to S_{\{0,1,n\}}\mathcal{D} \times_{S_{\{1,n\}}\mathcal{D}}^h S_{\{1,2,\ldots,n\}}\mathcal{D}$ is not an equivalence. Proposition 6.5 has an analogue with the same proof here as well.

\begin{proposition}
The $2$-Segal maps $f_\mathcal{T}$ corresponding to triangulations where some triangle does not contain 0 are not necessarily weak equivalences.
\end{proposition}

Because $iS_\bullet\mathcal{C}$ is a special case of $wS_\bullet\mathcal{C}$, and we have replaced the homotopy pullback with an ordinary pullback, we have a $wS_\bullet\mathcal{C}$ version of Proposition 8.6.

\begin{proposition}
The $2$-Segal maps $f_\mathcal{T}$ with source $wS_\bullet \mathcal{C}$ corresponding to triangulations where some triangle does not contain $0$ are not necessarily weak equivalences.
\end{proposition}

\section{A sufficient condition for $S_\bullet \mathcal{C}$ to be 2-Segal} \label{C8}
\thispagestyle{myheadings}

We now have some examples of possible strict left $2$-Segal spaces : $\text{iso}(s_\bullet\mathcal{C})$, $S_\bullet\mathcal{C}$, $wS_\bullet\mathcal{C}$, and $|iS_\bullet\mathcal{C}|$. Therefore, we now look for sufficient conditions under which some of these simplicial spaces are necessarily $2$-Segal. Analyzing our previous work we find the following criterion. 

\begin{proposition}
Let $\mathcal{C}$ be a category with cofibrations. Then $S_\bullet\mathcal{C}$ is $2$-Segal if any diagram in $\mathcal{C}$ of the form \[\begin{tikzcd}[sep = small]
	A && \bullet && X \\
	\\
	0 && B && Y
	\arrow[curve={height=-18pt}, tail, from=1-1, to=1-5]
	\arrow[two heads, from=1-5, to=3-5]
	\arrow[tail, from=3-3, to=3-5]
	\arrow[from=1-1, to=3-1]
	\arrow[tail, from=3-1, to=3-3]
\end{tikzcd}\] can be extended to \[\begin{tikzcd}[sep = small]
	A && C && X \\
	\\
	0 && B && Y
	\arrow[curve={height=-18pt}, tail, from=1-1, to=1-5]
	\arrow[two heads, from=1-5, to=3-5]
	\arrow[from=1-1, to=3-1]
	\arrow[tail, from=3-1, to=3-3]
	\arrow[tail, from=3-3, to=3-5]
	\arrow[dashed, tail, from=1-1, to=1-3]
	\arrow[dashed, tail, from=1-3, to=1-5]
	\arrow[two heads, from=1-3, to=3-3]
\end{tikzcd}\] where the squares are pushouts and the right square is a pullback. 
    
\end{proposition}

\begin{proof}
By \cite[2.3.2(4)]{DyckerhoffKapranov12} and since we argued in Section 8 that we can replace homotopy pullbacks in the $2$-Segal maps of $S_\bullet \mathcal{C}$ with ordinary pullbacks, it suffices to show that the $2$-Segal maps of the forms $S_n\mathcal{C} \to S_{\{0,j,j+1,\ldots, n\}}\mathcal{C} \times_{S_{\{0,j\}}\mathcal{C}}S_{\{0,1,\ldots j\}}\mathcal{C}$ and $S_n\mathcal{C} \to S_{\{0,1,\ldots j,n\}}\mathcal{C} \times_ {S_{\{j,n\}}\mathcal{C}} S_{\{j,j+1,\ldots,n\}}\mathcal{C}$ are equivalences. The $2$-Segal maps of the first of these two forms are equivalences by the work in Section 8. We focus on $2$-Segal maps of the second form. To show essential surjectivity of these it suffices to show that the diagram 
\[\begin{tikzcd}[row sep=1.5mm, column sep = 1.5mm]
	{A_{0,1}} && {A_{0,2}} && \cdots && {A_{0,j}} && \bullet && \bullet && \cdots && \bullet && {A_{0,n}} \\
	\\
	&& {A_{1,2}} && \cdots && {A_{1,j}} && \bullet && \bullet && \cdots && \bullet && {A_{1,n}} \\
	\\
	&&&&&& \vdots && \vdots && \vdots && \vdots && \vdots && \vdots \\
	\\
	&&&&&& {A_{j-1,j}} && \bullet && \bullet && \cdots && \bullet && {A_{j-1,n}} \\
	\\
	&&&&&&&& {A_{j,j+1}} && {A_{j, j+2}} && \cdots && {A_{j,n-1}} && {A_{j,n}} \\
	\\
	&&&&&&&&&& {A_{j+1,j+2}} && \cdots && {A_{j +1,n-1}} && {A_{j +1,n}} \\
	\\
	&&&&&&&&&&&&&& \vdots && \vdots \\
	\\
	&&&&&&&&&&&&&& {A_{n-2,n-1}} && {A_{n-2,n}} \\
	\\
	&&&&&&&&&&&&&&&& {A_{n-1,n}}
	\arrow[tail, from=1-1, to=1-3]
	\arrow[tail, from=1-3, to=1-5]
	\arrow[tail, from=1-5, to=1-7]
	\arrow[tail, from=3-3, to=3-5]
	\arrow[tail, from=3-5, to=3-7]
	\arrow[two heads, from=1-3, to=3-3]
	\arrow[two heads, from=1-7, to=3-7]
	\arrow[two heads, from=3-7, to=5-7]
	\arrow[two heads, from=5-7, to=7-7]
	\arrow[curve={height=-18pt}, tail, from=1-7, to=1-17]
	\arrow[two heads, from=1-17, to=3-17]
	\arrow[two heads, from=3-17, to=5-17]
	\arrow[two heads, from=5-17, to=7-17]
	\arrow[two heads, from=7-17, to=9-17]
	\arrow[curve={height=-18pt}, tail, from=3-7, to=3-17]
	\arrow[curve={height=-18pt}, tail, from=7-7, to=7-17]
	\arrow[tail, from=9-9, to=9-11]
	\arrow[tail, from=9-11, to=9-13]
	\arrow[tail, from=9-13, to=9-15]
	\arrow[tail, from=9-15, to=9-17]
	\arrow[two heads, from=9-11, to=11-11]
	\arrow[two heads, from=9-15, to=11-15]
	\arrow[two heads, from=9-17, to=11-17]
	\arrow[tail, from=11-11, to=11-13]
	\arrow[tail, from=11-13, to=11-15]
	\arrow[tail, from=11-15, to=11-17]
	\arrow[two heads, from=11-15, to=13-15]
	\arrow[two heads, from=11-17, to=13-17]
	\arrow[two heads, from=13-15, to=15-15]
	\arrow[two heads, from=13-17, to=15-17]
	\arrow[tail, from=15-15, to=15-17]
	\arrow[two heads, from=15-17, to=17-17]
\end{tikzcd}\] can be extended to an object of $S_n\mathcal{C}$ compatible with the given morphisms. We make this extension one object at a time. First, we use our hypothesis to extend \[\begin{tikzcd}[sep = small]
	{A_{0,j}} && \bullet && {A_{0,n}} \\
	\\
	0 && {A_{j,j+1}} && {A_{j,n}.}
	\arrow[two heads, from=1-1, to=3-1]
	\arrow[tail, from=3-1, to=3-3]
	\arrow[tail, from=3-3, to=3-5]
	\arrow[two heads, from=1-5, to=3-5]
	\arrow[curve={height=-18pt}, tail, from=1-1, to=1-5]
\end{tikzcd}\] The newly added object we denote by $A_{0,j +1}$. Next, we use our hypothesis to extend \[\begin{tikzcd}[sep = small]
	{A_{0,j+1}} && \bullet && {A_{0,n}} \\
	\\
	0 && {A_{j+1,j+2}} && {A_{j +1,n}.}
	\arrow[curve={height=-18pt}, tail, from=1-1, to=1-5]
	\arrow[two heads, from=1-1, to=3-1]
	\arrow[tail, from=3-1, to=3-3]
	\arrow[tail, from=3-3, to=3-5]
	\arrow[two heads, from=1-5, to=3-5]
\end{tikzcd}\] Proceeding in this way fill out the first row. Now we add in the rest of the required objects by taking cokernels of composites of morphisms in the first row. Compatibility with preexisting morphisms in the first $j$ rows of the diagram comes from the uniqueness part of the universal property for pushouts. That the last $n - j$ rows may come from choices of cokernels in the first row comes from the fact that if $A \rightarrowtail B \twoheadrightarrow D$, $(A \rightarrowtail B \rightarrowtail C) \twoheadrightarrow E$, and $D \rightarrowtail E \twoheadrightarrow F$ are cofibration sequences, then so is $B \rightarrowtail C \twoheadrightarrow F$. To show fullness, it suffices to show that \[\begin{tikzcd}[sep = small]
	& {B_{0,j}} &&&& {B_{0,n}} \\
	{A_{0j}} &&&& {A_{0n}} \\
	& 0 && {B_{j,j+1}} && {B_{j,n}} \\
	0 && {A_{j,j+1}} && {A_{j,n}}
	\arrow[from=2-1, to=1-2]
	\arrow[from=2-5, to=1-6]
	\arrow[two heads, from=1-6, to=3-6]
	\arrow[two heads, from=2-5, to=4-5]
	\arrow[tail, from=4-3, to=4-5]
	\arrow[from=4-5, to=3-6]
	\arrow[from=4-3, to=3-4]
	\arrow[tail, from=3-4, to=3-6]
	\arrow[from=2-1, to=4-1]
	\arrow[from=1-2, to=3-2]
	\arrow[tail, from=3-2, to=3-4]
	\arrow[tail, from=4-1, to=4-3]
	\arrow[from=4-1, to=3-2]
	\arrow[curve={height=-18pt}, tail, from=1-2, to=1-6]
	\arrow[curve={height=-18pt}, tail, from=2-1, to=2-5]
\end{tikzcd}\] can be extended to \[\begin{tikzcd}[sep = small]
	& {B_{0,j}} && {B_{0,j+1}} && {B_{0,n}} \\
	{A_{0j}} && {A_{0, j+1}} && {A_{0n}} \\
	& 0 && {B_{j,j+1}} && {B_{j,n}} \\
	0 && {A_{j,j+1}} && {A_{j,n}}
	\arrow[from=2-1, to=1-2]
	\arrow[from=2-5, to=1-6]
	\arrow[two heads, from=1-6, to=3-6]
	\arrow[two heads, from=2-5, to=4-5]
	\arrow[tail, from=4-3, to=4-5]
	\arrow[from=4-5, to=3-6]
	\arrow[from=4-3, to=3-4]
	\arrow[tail, from=3-4, to=3-6]
	\arrow[from=2-1, to=4-1]
	\arrow[from=1-2, to=3-2]
	\arrow[tail, from=3-2, to=3-4]
	\arrow[tail, from=4-1, to=4-3]
	\arrow[from=4-1, to=3-2]
	\arrow[curve={height=-18pt}, tail, from=1-2, to=1-6]
	\arrow[curve={height=-18pt}, tail, from=2-1, to=2-5]
	\arrow[from=2-3, to=1-4]
	\arrow[tail, from=2-1, to=2-3]
	\arrow[tail, from=2-3, to=2-5]
	\arrow[two heads, from=2-3, to=4-3]
	\arrow[tail, from=1-2, to=1-4]
	\arrow[tail, from=1-4, to=1-6]
	\arrow[two heads, from=1-4, to=3-4]
\end{tikzcd}\] where the extension on the front and back faces comes from applying the hypothesis. This extension is obtained by applying the hypothesis to the front and back faces and then using that the back right square is a pullback square. We must check commutativity of the three squares containing the morphism $A_{0,j+1} \to B_{0,j+1}$. The top right square and vertical middle square commute because the morphism $A_{0,j +1} \to B_{0,j+1}$ was induced by the commutative square \[\begin{tikzcd}[sep = small]
	{A_{0,j+1}} & {A_{0,n}} & {B_{0,n}} \\
	{A_{j,j+1}} \\
	{B_{j,j+1}} && {B_{j,n}}
	\arrow[tail, from=1-1, to=1-2]
	\arrow[from=1-2, to=1-3]
	\arrow[tail, from=3-1, to=3-3]
	\arrow[two heads, from=1-3, to=3-3]
	\arrow[two heads, from=1-1, to=2-1]
	\arrow[from=2-1, to=3-1]
\end{tikzcd}.\] Commutativity of the top left square comes from the uniqueness part of the universal property of the pullback: $A_{0,j} \to B_{0,j} \rightarrowtail B_{0,j+1}$ and $A_{0,j} \rightarrowtail A_{0,j+1} \to B_{0,j+1}$ both can both be added to give a commutative diagram \[\begin{tikzcd}[sep=small]
	{A_{0,j}} && {B_{0,j}} \\
	\\
	0 && {B_{0,j+1}} &&& {B_{0,n}} \\
	\\
	\\
	&& {B_{j,j+1}} &&& {B_{j,n}}
	\arrow[tail, from=6-3, to=6-6]
	\arrow[two heads, from=3-6, to=6-6]
	\arrow[tail, from=3-3, to=3-6]
	\arrow[two heads, from=3-3, to=6-3]
	\arrow[dashed, from=1-1, to=3-3]
	\arrow[from=1-1, to=1-3]
	\arrow[tail, from=1-3, to=3-6]
	\arrow[two heads, from=1-1, to=3-1]
	\arrow[tail, from=3-1, to=6-3]
\end{tikzcd}.\] The uniqueness part of the universal property of pullbacks and pushouts gives faithfulness. 
\end{proof}

\begin{example}
The category $\mathcal{R}_f$ satisfies the hypothesis of Proposition 9.1. The missing object in \[\begin{tikzcd}[sep = small]
	A && \bullet && X \\
	\\
	0 && B && X/A
	\arrow[from=1-1, to=3-1]
	\arrow[tail, from=3-1, to=3-3]
	\arrow["g", tail, from=3-3, to=3-5]
	\arrow["q"', two heads, from=1-5, to=3-5]
	\arrow["f", curve={height=-18pt}, tail, from=1-1, to=1-5]
\end{tikzcd}\] can be filled in with $q^{-1}(g(B))$ which inherits its CW structure from $X$. The object $A$ is a subcomplex of $q^{-1}(g(B))$ since it is a subcomplex of $X$ that $q$ maps to the basepoint of $X/A$ which is contained in $g(B)$. The needed map from $q^{-1}(g(B))$ to $B$ is $g^{-1}q$ where $g^{-1}$ is defined on $g(B)$ and $q$ is restricted to $q^{-1}(g(B))$. We have a commutative square of the right shape and it remains to check that the left square is a pushout and the right square is biCartesian. The middle vertical map is the restriction of the map which quotients by $A$ so the left square is a pushout. A model for the pushout of $B \leftarrow q^{-1}(g(B)) \rightarrowtail X$ is $B \sqcup_{x \sim g^{-1}q(x)} X$. The process of quotienting by this relation can be described in two steps. First $A \subset X$ gets identified with the point $A \in X/A$. Next $q^{-1}(g(B))$ is bijectively identified with $B - \{A\}$. The resulting space is $X/A$. The right square is pullback since
given an extension problem of the form \[\begin{tikzcd}[sep = small]
	Z \\
	& {q^{-1}(g(B))} && X \\
	\\
	& B && {X/A}
	\arrow["h", curve={height=18pt}, from=1-1, to=4-2]
	\arrow["k"', curve={height=-18pt}, from=1-1, to=2-4]
	\arrow["q"', two heads, from=2-4, to=4-4]
	\arrow["g", tail, from=4-2, to=4-4]
	\arrow[tail, from=2-2, to=2-4]
	\arrow[two heads, from=2-2, to=4-2]
	\arrow[dashed, from=1-1, to=2-2]
\end{tikzcd}\] the unique solution is given by $k$ with restricted codomain. The morphism $k$ has image in $q^{-1}(g(B))$ since
$q(k(Z)) \subset g(B)$ by commutativity. 

As a special case, we get that the category of based finite sets with injective maps as cofibrations satisfies the hypothesis of Proposition 9.1. In fact, the arguments of this example go through for the category of based sets with cofibrations the injective maps in this category. 
\end{example}


\begin{remark}
Proto-exact categories with admissible monomorphisms as cofibrations as defined in \cite[2.4.2]{DyckerhoffKapranov12} are close to satisfying the hypothesis of Proposition 9.1. The part about the left square being a pushout is where this hypothesis conceivably might fail for a proto-exact category. Conversely, a category with cofibrations need not be exact if we say the admissible monomorphisms are cofibrations and their cokernels are admissible epimorphisms. This is because the cokernels of cofibrations do not need to be closed under composition. 
\end{remark}

\appendix
\section{Upper and lower 2-Segal}
\thispagestyle{myheadings}

Here we prove that what we have been calling left and right $2$-Segal are the same as lower and upper $2$-Segal respectively. 

\begin{definition}
Let $\mathcal{D}$ be a model category, $X \colon \Delta^{op} \to \mathcal{D}$ a simplicial object, and $0 < i < n$. Consider the squares \[\begin{tikzcd}[sep = small]
	{X_{n + 1}} && {X_n} &&&& {X_{n + 1}} && {X_n} \\
	&&&& {\text{and}} \\
	{X_n} && {X_{n - 1}} &&&& {X_n} && {X_{n - 1}}
	\arrow["{d_{n + 1}}", from=1-7, to=1-9]
	\arrow["{d_i}"', from=1-7, to=3-7]
	\arrow["{d_n}"', from=3-7, to=3-9]
	\arrow["{d_i.}", from=1-9, to=3-9]
	\arrow["{d_0}", from=1-1, to=1-3]
	\arrow["{d_0}", from=3-1, to=3-3]
	\arrow["{d_i}", from=1-3, to=3-3]
	\arrow["{d_{i + 1}}"', from=1-1, to=3-1]
\end{tikzcd}\] We say $X$ is \emph{upper} $2$-Segal if for 
any choices of $i$ and $n$, the left square is homotopy cartesian. We say $X$ is \emph{lower} $2$-Segal if for any choices of $i$ and $n$, the right square is homotopy cartesian. 
\end{definition}

\remark{We use the formulation that appears in \cite[1]{Feller19} but the first place the definition appears in the literature is \cite[2.2]{Poguntke17}.

\begin{proposition}
A simplicial object in a model category is upper $2$-Segal if and only if it is right $2$-Segal, and it is lower $2$-Segal if and only if it is left $2$-Segal. 
\end{proposition}

\begin{proof}
We prove the first statement; the second statement is dual. It is known that a simplicial object $Y$ is lower $2$-Segal if and only if for each $n$ the map 
$Y_n \to Y_{\{0,1,n\}} \times^h_{Y_{\{1,n\}}} Y_{\{1,2,\ldots,n\}}$ is a weak equivalence. So if $X$ is right $2$-Segal then by duality and definition we can conclude that $X$ is upper $2$-Segal. Now we turn to the converse.

We can restate our definition of right $2$-Segal in terms of homotopy pullback squares as follows. A simplicial object $X$ in a model category is right $2$-Segal exactly when the diagram \[\begin{tikzcd}[sep = small]
	{X_k} && {X_{j+1}} \\
	\\
	{X_{k - j}} && {X_1}
	\arrow["{(d_{j + 1})^{k - j - 1}}", from=1-1, to=1-3]
	\arrow["{(d_0)^j}"', from=1-1, to=3-1]
	\arrow["{(d_1)^{k - j - 1}}"', from=3-1, to=3-3]
	\arrow["{(d_0)^j}", from=1-3, to=3-3]
\end{tikzcd}\] is homotopy cartesian for all $0 < j < k$. 

Assume that $X$ is upper $2$-Segal. Then

\[\begin{tikzcd}[sep = small]
	{X_3} && {X_2} \\
	\\
	{X_2} && {X_1}
	\arrow["{d_2}", from=1-1, to=1-3]
	\arrow["{d_0}"', from=1-1, to=3-1]
	\arrow["{d_0}", from=1-3, to=3-3]
	\arrow["{d_1}"', from=3-1, to=3-3]
\end{tikzcd}\] is homotopy cartesian by assumption. In other words $X_3 \simeq X_{\{0,1,3\}} \times^h_{X_{\{1,3\}}} X_{\{1,2,3\}}$. By assumption, we have homotopy cartesian squares

\[\begin{tikzcd}[sep = small]
	{X_4} && {X_3} && {X_3} && {X_2} \\
	&&& {\text{and}} \\
	{X_3} && {X_2} && {X_2} && {X_1}
	\arrow["{d_2}", from=1-1, to=1-3]
	\arrow["{d_0}"', from=1-1, to=3-1]
	\arrow["{d_0}", from=1-3, to=3-3]
	\arrow["{d_1}"', from=3-1, to=3-3]
	\arrow["{d_2}", from=1-5, to=1-7]
	\arrow["{d_0}"', from=1-5, to=3-5]
	\arrow["{d_1}"', from=3-5, to=3-7]
	\arrow["{d_0}", from=1-7, to=3-7]
\end{tikzcd}\] so that  

\[\begin{tikzcd}[sep = small]
	{X_4} && {X_2} \\
	\\
	{X_3} && {X_1}
	\arrow["{d_2d_2}", from=1-1, to=1-3]
	\arrow["{d_0}"', from=1-1, to=3-1]
	\arrow["{d_0}", from=1-3, to=3-3]
	\arrow["{d_1d_1}"', from=3-1, to=3-3]
\end{tikzcd}\] is homotopy cartesian. Similarly, we have the homotopy cartesian square \[\begin{tikzcd}[sep = small]
	{X_4} && {X_2} \\
	\\
	{X_3} && {X_1.}
	\arrow["{d_0d_0}", from=1-1, to=1-3]
	\arrow["{d_0d_0}"', from=3-1, to=3-3]
	\arrow["{d_3}"', from=1-1, to=3-1]
	\arrow["{d_1}", from=1-3, to=3-3]
\end{tikzcd}\]

For $n \leq 5$ we have verified that for $0 < j < k < n$, the square \[\begin{tikzcd}[sep = small]
	{X_k} && {X_{j+1}} \\
	\\
	{X_{k - j}} && {X_1}
	\arrow["{(d_{j + 1})^{k - j - 1}}", from=1-1, to=1-3]
	\arrow["{(d_0)^j}"', from=1-1, to=3-1]
	\arrow["{(d_1)^{k - j - 1}}"', from=3-1, to=3-3]
	\arrow["{(d_0)^j}", from=1-3, to=3-3]
\end{tikzcd}\] is homotopy cartesian. To use induction, assume that for some $n \geq 5$ the above square is homotopy cartesian for all $0 < j < k < n$. Because $X$ is upper $2$-Segal, \[\begin{tikzcd}[sep = small]
	{X_n} && {X_{n - 1}} \\
	\\
	{X_{n - 1}} && {X_{n - 2}}
	\arrow["{d_0}", from=1-1, to=1-3]
	\arrow["{d_3}"', from=1-1, to=3-1]
	\arrow["{d_0}"', from=3-1, to=3-3]
	\arrow["{d_2}", from=1-3, to=3-3]
\end{tikzcd}\] is homotopy cartesian. By inductive assumption, the square \[\begin{tikzcd}[sep = small]
	{X_{n - 1}} && {X_2} \\
	\\
	{X_{n - 2}} && {X_1}
	\arrow["{(d_2)^{n - 3}}", from=1-1, to=1-3]
	\arrow["{d_0}"', from=1-1, to=3-1]
	\arrow["{(d_1)^{n - 3}}"', from=3-1, to=3-3]
	\arrow["{d_0}", from=1-3, to=3-3]
\end{tikzcd}\] is homotopy cartesian so we get the homotopy cartesian square \[\begin{tikzcd}[sep = small]
	{X_n} && {X_{n - 1}} \\
	\\
	{X_2} && {X_1.}
	\arrow["{d_0}", from=1-1, to=1-3]
	\arrow["{(d_2)^{n - 2}}"', from=1-1, to=3-1]
	\arrow["{(d_1)^{n - 2}}", from=1-3, to=3-3]
	\arrow["{d_0}"', from=3-1, to=3-3]
\end{tikzcd}\] If $j \neq 1$, then by inductive assumption the square \[\begin{tikzcd}[sep = small]
	{X_{n - 1}} && {X_j} \\
	\\
	{X_{n - j}} && {X_1}
	\arrow["{(d_j)^{n - j - 1}}", from=1-1, to=1-3]
	\arrow["{(d_0)^{j - 1}}"', from=1-1, to=3-1]
	\arrow["{(d_1)^{n - j - 1}}"', from=3-1, to=3-3]
	\arrow["{(d_0)^{j - 1}}", from=1-3, to=3-3]
\end{tikzcd}\] is homotopy cartesian. By inductive assumption the square \[\begin{tikzcd}[sep = small]
	{X_{j + 1}} && {X_2} \\
	\\
	{X_j} && {X_1}
	\arrow["{d_0}"', from=1-1, to=3-1]
	\arrow["{(d_2)^{j - 1}}", from=1-1, to=1-3]
	\arrow["{d_0}", from=1-3, to=3-3]
	\arrow["{(d_1)^{j-1}}"', from=3-1, to=3-3]
\end{tikzcd}\] is homotopy cartesian. Putting the previous three squares together we have the commutative diagram \[\begin{tikzcd}
	&&&& {X_n} \\
	\\
	&& {X_{j+1}} &&&& {X_{n-1}} \\
	\\
	{X_2} &&&& {X_j} &&&& {X_{n - j}.} \\
	\\
	&& {X_1} &&&& {X_1}
	\arrow["{(d_{j+1})^{n - j - 1}}"', from=1-5, to=3-3]
	\arrow["{d_0}", from=1-5, to=3-7]
	\arrow["{d_0}"', from=3-3, to=5-5]
	\arrow["{(d_j)^{n - j - 1}}", from=3-7, to=5-5]
	\arrow["{(d_2)^{j - 1}}"', from=3-3, to=5-1]
	\arrow["{(d_1)^{j - 1}}", from=5-5, to=7-3]
	\arrow["{d_0}"', from=5-1, to=7-3]
	\arrow["{(d_0)^{j - 1}}"', from=5-5, to=7-7]
	\arrow["{(d_0)^{j - 1}}", from=3-7, to=5-9]
	\arrow["{(d_1)^{n - j - 1}}", from=5-9, to=7-7]
\end{tikzcd}\] The left and center squares together form a homotopy cartesian rectangle, and the left and right squares are homotopy cartesian, so the rectangle formed by the right and center squares is homotopy cartesian. By induction we conclude that $X$ is right $2$-Segal. 
  
\end{proof}

\end{document}